\documentclass[11pt]{amsart}
\usepackage[all]{xy}
\usepackage[utf8x]{inputenc}
\usepackage{amssymb, amsmath,epsfig,url}
\usepackage[lmargin=1.1in,rmargin=1.1in,tmargin=1.1in,bmargin=1.1in]{geometry}
\newcommand{\color}[6]{}

\input xy
\xyoption{all}
\usepackage {amsmath}
\usepackage{amssymb}
\usepackage{amsthm}
\usepackage {amsfonts}
\usepackage{graphicx}
\usepackage{hyperref}
\usepackage{url}

\theoremstyle{plain}
\newtheorem{thm}{Theorem}[subsection]
\newtheorem{conj}[thm]{Conjecture}
\newtheorem{prop}[thm]{Proposition}
\newtheorem{cor}[thm]{Corollary}
\newtheorem{lem}[thm]{Lemma}

\newtheorem{defn}[thm]{Definition}
\theoremstyle{definition}
\newtheorem{rem}{Remark}[subsection]
\newtheorem{examp}[rem]{Example}

\newcommand{\Cp}{\mathbb{C}}
\DeclareMathOperator{\op}{op}
\DeclareMathOperator{\cyc}{cyc}
\DeclareMathOperator{\Ho}{H}
\DeclareMathOperator{\HH}{HH}
\DeclareMathOperator{\HC}{HC}
\DeclareMathOperator{\Ext}{Ext}
\DeclareMathOperator{\RHom}{RHom}
\DeclareMathOperator{\Ch}{C}
\DeclareMathOperator{\const}{\mathbf{const}}
\DeclareMathOperator{\Set}{\mathbf{Set}}
\DeclareMathOperator{\SSet}{\mathbf{SSet}}
\DeclareMathOperator{\Hom}{Hom}
\DeclareMathOperator{\simp}{simp}
\DeclareMathOperator{\Top}{\mathbf{Top}}
\DeclareMathOperator{\Fun}{Fun}
\DeclareMathOperator{\No}{N}
\DeclareMathOperator{\red}{red}
\DeclareMathOperator{\DD}{D}
\DeclareMathOperator{\Ker}{Ker}
\DeclareMathOperator{\cpct}{cpct}
\DeclareMathOperator{\Map}{Map}
\DeclareMathOperator{\Der}{Der}
\DeclareMathOperator{\DDer}{\mathbb{D}er}
\DeclareMathOperator{\DR}{DR}
\DeclareMathOperator{\ncm}{\mu_{nc}}
\DeclareMathOperator{\ncml}{\tilde{\mu}_{nc}}
\DeclareMathOperator{\bp}{\mathbf{bp}}
\DeclareMathOperator{\res}{\mathbf{res}}
\DeclareMathOperator{\Zz}{Z}
\DeclareMathOperator{\Ss}{S}

\usepackage{amssymb, amsmath, epsfig, graphics, amscd}
\usepackage{amssymb, amsmath, epsfig, graphics, amscd}
\date {}
\author{Ben Davison}
\title{Superpotential algebras and manifolds}

\begin{document}
\maketitle
\begin{abstract}
We study a special class of Calabi-Yau algebras (in the sense of Ginzburg): those arising as the fundamental group algebras of acyclic manifolds.  Motivated partly by the usefulness of `superpotential descriptions' in motivic Donaldson-Thomas theory, we investigate the question of whether these algebras admit superpotential presentations.  We establish that the fundamental group algebras of a wide class of acyclic manifolds, including all hyperbolic manifolds, do not admit such descriptions, disproving a conjecture of Ginzburg regarding them.  We also describe a class of manifolds that do admit such descriptions, and discuss a little their motivic Donaldson-Thomas theory.  Finally, some links with topological field theory are described.\end{abstract}
\setcounter{tocdepth}{1}
\tableofcontents
\section{Introduction}
\subsection{Background}
Since the foundational paper \cite{Ginz} it has been understood that if a $d$-dimensional compact orientable manifold has a contractible universal cover, its fundamental group algebra is a `Calabi-Yau algebra' of dimension $d$.  Indeed to topologists this appears to be something of a `folklore result'.  It has also been understood for some time that a large number of the Calabi-Yau algebras that one meets have a special form -- they are a kind of noncommutative symplectic reduction.  In 3 dimensions, algebras of this form are known as Jacobi algebras associated to superpotential algebras -- they play an important role in supersymmetric gauge theory.  Precisely, the superpotential algebra is a kind of smooth differential graded algebra, with the Jacobi algebra as its homology algebra, concentrated in degree zero.  In 3 dimensions we have a result of Bocklandt \cite{CY3dim} that states that a positively graded Calabi-Yau algebra that is given by a quiver with relations is necessarily a Jacobi algebra.  More recently, Van den Bergh has shown in \cite{VdB10} that formal Calabi-Yau algebras (in every dimension) are always quasi-isomorphic to `deformed preprojective algebras,' which are a kind of superpotential algebra in higher dimensions.
\smallbreak
This all makes it fairly natural to ask whether, or maybe even assume that, the fundamental group algebras above are superpotential algebras.  Indeed a specific form for the superpotential presentation of the fundamental group algebra of a 3-dimensional acyclic manifold is conjectured in (\cite{Ginz} Conjecture 6.2.1).  One of the main results of this paper is that this conjecture is false in general in all dimensions greater than or equal to 2 -- it appears to be the case that only quite special manifolds can admit presentations as superpotential algebras.  It turns out that there is a topological obstruction to being able to give such a presentation, which translates to a property of the fundamental group ring.  One of our main results is that the fundamental group algebra of no hyperbolic manifold is a superpotential algebra.
\smallbreak
One of the primary purposes of investigating this question was to understand the motivic Donaldson-Thomas invariants of (acyclic) 3-manifolds.  These are much easier to define and calculate in the case in which the fundamental group algebra is quasi-isomorphic to a superpotential algebra, and so finding out when this is the case is in some way a preliminary step in understanding these invariants, or perhaps just a shortcut.  We offer a little more discussion of the 3-dimensional case at the end of the paper, as well as a partial converse to the negative results contained in the rest of the paper.  In particular, we give a constructive proof that fundamental group algebras of acyclic 3-dimensional trivial circle bundles are the Jacobi algebras associated to superpotential algebras.

\subsection{Structure of the paper}
In Section 2 we introduce some of the basic tools from Algebra that will appear throughout the paper.  It is here that we meet the definition of a `exact\footnote{This terminology was suggested to me by Maxim Kontsevich, it is what is elsewhere called `strongly Calabi-Yau'.} Calabi-Yau algebra', due to Keller.  The proof of the main result rests heavily on this notion.  
\smallbreak
We try to assume only a limited topological background for this paper.  In keeping with this, in Section 3 we give a reasonably self-contained introduction to all the topological notions we will need, especially equivariant homology and simplicial sets.
\smallbreak
In Section 4 we will see the necessary Noncommutative Geometry required to state the construction of a superpotential algebra, and prove that such algebras are exact Calabi-Yau.  This is our means for proving that fundamental group algebras of hyperbolic manifolds are not quasi-isomorphic to superpotential algebras, since we will see that these algebras are not exact Calabi-Yau.  
\smallbreak
In Section 5 we consider the fundamental group algebras of acyclic manifolds, and prove some homological properties regarding them.  In particular, it is here that we offer a complete proof that they are Calabi-Yau algebras.  
\smallbreak
In Section 6 we consider the Connes long exact sequence associated to a manifold $M$, and find obstructions to fundamental group algebras being superpotential algebras.  In particular, we will prove, by analyzing the topological interpretation of the Connes long exact sequence, that the Calabi-Yau structure of Section 5 can never be exact -- so in order for $k[\pi_1(M)]$ to be exact Calabi-Yau there must be more than 1 Calabi-Yau structure on it, which will correspond to a statement about central units in $k[\pi_1(M)]$.  
\smallbreak
We finish with a brief discussion of the case of hyperbolic manifolds, and prove that for all hyperbolic manifolds $M$ of dimension greater than 1, $k[\pi_1(M)]$ is not a superpotential algebra.  This will follow from the fact that $k[\pi_1(M)]$ has trivial centre, and so only 1 possible Calabi-Yau structure - the one that we have proved is \textit{not} exact.  
\smallbreak
Section 7 concerns the Donaldson-Thomas theory of 3-manifolds, and also contains a brief discussion regarding the links between the present work and topological field theory.

\subsection{Acknowledgements}
This project began with an attempt to understand Section 6 of \cite{Ginz}, the foundational paper, the pervasive influence of which will become apparent upon reading the rest of this paper.  While the present work owes a lot to \cite{Ginz}, there is also a strong influence of subsequent work by Keller and Van den Bergh, especially \cite{VdB10}, \cite{CYTC}, \cite{kel09}.  Over the years I have had the fortune of a great many fruitful conversations regarding the subject matter, it would actually take quite a long time to thank everyone who has had something helpful to say.  Special mention should, however, go to Kevin Costello, for pointing me towards \cite{Lur09} and \cite{CostTFT}.  I thank Jacob Lurie for pointing out the relation between TFTs and Calabi-Yau algebras, Graeme Segal, Ezra Getzler, Jeff Giansiracusa and Oscar Randal-Williams for conversations about topology, Michel Van den Bergh (who also suggested improvements to an earlier version) and Bernhard Keller for conversations about algebra, Alastair King and Victor Ginzburg for conversations about Koszul duality, Richard Wade and Dawid Kielak for conversations about geometric group theory, and David Craven for conversations about group theory. Further thanks must also go to Victor Ginzburg, Maxim Kontsevich, Michel Van den Bergh and Bernhard Keller for their insightful comments and corrections of earlier drafts.  Finally, special thanks go to Bal\'{a}zs Szendr\H{o}i for his constant support and mathematical help throughout the last four years.

\section{Definitions and notation}
\subsection{Standing conventions}
We will be primarily concerned with properties of $k$-algebras associated to manifolds, where $k$ will always be a field; the salient influence of the characteristic of $k$ on these properties is that in characteristic 2 we can extend a few results to \textit{non-orientable} manifolds.  This is an artefact of the fact that non-orientable manifolds \textit{do} have an orientation if we take coefficients in a field of characteristic 2.  Throughout the paper, unless otherwise stated, $k$ will denote a field of characteristic zero, and we will be working with differential graded $k$-unital algebras.  We assume always that our algebras are concentrated solely in nonpositive or nonnegative degree, and that our differential increases degree in the first case, and decreases it in the second, so that a $k$-unital differential graded algebra is just a unital differential graded algebra $A$ and an injective (after taking homology) morphism of unital graded algebras $k\rightarrow A$, where $k$ is considered a graded algebra concentrated in degree zero.
\medbreak
We warn the reader at the outset that there is a potential for confusion regarding the switch from homological grading, for which differentials will decrease degree, to cohomological grading, for which differentials will increase degree.  Objects such as the singular chains in spaces will be homologically graded, but more algebraic objects coming from noncommutative geometry will be cohomologically graded.  The result is that when objects from topology are imported into noncommutative geometry the degree of simplicial chains changes by a sign.
\medbreak
\begin{rem}
Our treatment is slightly less general than that of \cite{VdB10}, in that we work always over our field $k$, where perhaps one may want the extra flexibility of working over a finite dimensional semisimple $k$-algebra $l$.  It is possible to extend the technical background to this context (see \cite{VdB10}) at the cost of increasing the risk of making silly technical mistakes (see an earlier draft of this paper).  We remark here that, in the area in which we are interested, it is unlikely that $l$ will ever be anything other than $k$.  This is because we are interested in group algebras, $k[G]$, for $G$ containing no torsion.  It is a conjecture of Kaplansky that there are no nontrivial zero divisors in such a ring, implying that there are no nontrivial idempotents (see \cite{Cli80} for a partial solution, but also \cite{RolfZhu98} and \cite{Pus02} for solutions particularly interesting from the viewpoint of this paper, in which we are concerned with groups that have classifying spaces homotopic to manifolds).  This is an open problem that dates back to the 1940s - one does not expect to just run into counterexamples.
\end{rem}
In most cases we will also have the extra structure of an augmentation, i.e. a morphism of differential graded algebras $A\rightarrow k$ such that the composition
\[
k\rightarrow A \rightarrow k
\]
is equal to the identity.\medbreak
Where tensor products appear unadorned, they are to be taken over $k$.  All tensor products are derived tensor products (though of course this makes no difference for unadorned tensor products).\medbreak
If $M$ is a bimodule over an algebra $D$, then we define the differential graded vector space $M_D:=M/[D,M]$.  For an arbitrary algebra $D$ we define the differential graded vector space $D_{\cyc}:=D/[D,D]$, the supercommutator quotient of $D$.
\medbreak
Given a differential graded algebra $D$ over $k$, we define
\[
D^e:=D\otimes D^{\op}.
\]
The category of $D$-bimodules is naturally identified with the category of left $D^e$-modules, via the natural isomorphism.  There is also a natural equivalence of categories between the categories of right and left $D^e$-modules, thanks to the natural isomorphism $D^e\cong (D^e)^{\op}$.  We consider $D\otimes D$ as a $D^e$-\textit{bi}module, with left action given by the outer bimodule structure, and right action given by the inner.  We define the functor
\[
-^{\vee}=\Hom_{D^e}(-,D\otimes D).
\]
By the above comments this is an endofunctor for the category of $D$-bimodules.
\subsection{Cyclic and Hochschild homology of unital differential graded algebras}
Let $A$ be a unital differential graded algebra.  We remind the reader of the basic facts regarding Hochschild and (ordinary) cyclic homology.  An excellent and highly readable reference for this is \cite{Lod98}.  Firstly, we start with the complex (\textit{the bar complex}):
\[
B(A)=\bigoplus_{i\geq 2} A^{\otimes i}.
\]
The differential on $A$ induces a differential on $B(A)$.  $B(A)$ carries another differential, given by
\[
\delta(a_0,..., a_n)=\sum_{0\leq i\leq n-1}(-1)^i(a_0,..., a_ia_{i+1},..., a_n).
\]
Taking the (sum) total complex of the double complex we obtain a semifree $A^e$-module resolution of the diagonal bimodule $A$.  It follows that we can use this resolution to calculate $\HH_*(A):=A\otimes_{A^e} A$, the \textit{Hochschild homology} of $A$.  We obtain a complex
\[
C(A)=\bigoplus_{i\geq 1} A^{\otimes i}
\]
which again inherits one differential from $A$.  The second differential is given by
\[
\delta(a_0,..., a_n)=\sum_{0\leq i\leq n-1}(-1)^i(a_0,..., a_ia_{i+1},..., a_{n})+(-1)^n(a_na_0, a_1,..., a_{n-1}).
\]
Taking the sum total complex again, we obtain a complex calculating $\HH_{*}(A)$, whose differential we denote $b$.\smallbreak
There is a natural map of complexes $C(A)\rightarrow \overline{C}(A)$, where 
\[
\overline{C}(A):=\overline{A}\oplus (A\oplus \overline{A})\oplus(A\oplus \overline{A}^{\otimes 2})\oplus\ldots,
\]
with $\overline{A}:=A/k$, is the \textit{normalized} Hochschild complex, with differential $\overline{b}$ defined in the same way as $b$.  This is a quasi-isomorphism.\medbreak
There is a second differential on the normalized Hochschild complex $\overline{C}(A)$, which is denoted $B$ (the Connes differential).  It is defined by
\[
\overline{B}(a_0,...,a_n):=\sum_{0\leq i\leq n}(-1)^{ni}(1,a_i,...,a_n,a_0,...,a_{i-1}).
\]
For example we have $\overline{B}(a_0)=(1,a_0)$.  Clearly $\overline{B}$ satisfies $\overline{B}^2=0$.  It also satisfies $\overline{b}\overline{B}+\overline{B}\overline{b}=0$, so we may take the double complex $\overline{CC}(A)$ of complexes
\[
\xymatrix{
A\otimes \overline{A}\otimes \overline{A}\ar[d]^{\overline{\delta}} & A\otimes\overline{A}\ar[d]^{\overline{\delta}} \ar[l]^-{\overline{B}} &\overline{A}\ar[l]^-{\overline{B}} \\
A\otimes\overline{A}\ar[d]^{\overline{\delta}} & \overline{A}\ar[l]^-{\overline{B}}\\
\overline{A}.
}
\]
We define $\HC_{*}(A)$ to be the total homology of the triple complex.  After endowing $\overline{C}(A)$ with a third differential, identically zero, there is an inclusion of triple complexes $\overline{C}(A)\rightarrow \overline{CC}(A)$, whose mapping cone is naturally quasi-isomorphic to $\overline{CC}(A)[2]$.  It follows that there is a long exact sequence of homology
\[
\rightarrow \HC_{n+1}(A)\rightarrow \HC_{n-1}(A)\rightarrow \HH_n(A)\rightarrow \HC_n(A)\rightarrow.
\]
The data $(\overline{C}(A),\overline{b},\overline{B})$ is an example of a \textit{mixed complex}.  These are triples $(C,b,B)$ such that $b^2=0$, $B^2=0$, $bB+Bb=0$, the degree of $b$ is 1, and the degree of $B$ is -1.  Morphisms of mixed complexes are defined in the obvious way, while quasi-isomorphisms are morphisms that induce isomorphisms after taking homology with respect to $b$.  For any mixed complex we can form the above double complex, and define its Hochschild and cyclic homology as before (the Hochschild homology is just homology with respect to $b$).  A quasi-isomorphism $A\rightarrow A'$ of differential graded algebras induces a quasi-isomorphism of mixed complexes
\[
(\overline{C}(A),\overline{b},\overline{B})\rightarrow (\overline{C}(A'),\overline{b'},\overline{B'}).
\]
A quasi-isomorphism of mixed complexes in turn induces a quasi-isomorphism of Hochschild and cyclic homology.  There is a quasi-isomorphism
\[
(\overline{C}(A),\overline{b},\overline{B})\rightarrow (C(A),b,B)
\]
where $B$ is the Connes differential
\[
B(a_0,...,a_n)=\sum_{0\leq i\leq n}((-1)^{ni}(1,a_i,...,a_n,a_0,...,a_{i-1})+(-1)^{n(i+1)}(a_i,1,a_{i+1},...,a_n,a_0,...,a_{i-1})).
\]
This gives rise to a different triple complex $CC(A)$, and by the above comments, we may just as well compute cyclic homology as the total homology of this complex.
\subsection{Calabi-Yau conditions}
In this section we present the original definition of a Calabi-Yau algebra, along with a strengthening of it due to Keller.  The point of this strengthening, for us, is that it provides a tool to study obstructions to being able to give a Calabi-Yau algebra a `superpotential presentation'.
\begin{defn}
An algebra $A$ is \textit{homologically finite} if it is perfect as an $A^e$-module.
\end{defn}
There is a natural quasi-isomorphism of graded vector spaces, for $M$ and $N$ objects of the category of perfect modules over $A^e$:
\[
\xymatrix{
\RHom_{A^e}(M,N)\ar[r]^-{\sim}& M^{\vee}\otimes_{A^e} N.
}
\]
In particular, if $A$ is homologically finite, there is an isomorphism
\[
\HH_d(A)\cong \Ext_{A^e}^{-d}(A^{\vee},A).
\]
\begin{defn}
We say that $A$ is Calabi-Yau of dimension $d$ if it is homologically finite and there is a self-dual isomorphism
\[
f:A^{\vee}[d]\rightarrow A.
\]
\end{defn}
The self-duality is expressed by the fact that $f^{\vee}[d]$ is again a map from $A^{\vee}[d]$ to $A$, and we ask that $f^{\vee}[d]=f$.  A map $f$ as above can be considered as an object in $\HH_d(A)$.  Self-duality amounts to requiring that $f$ be fixed under the flip isomorphism, that is, let $\beta$ act on $A\otimes_{A^e} A$ by flipping the two copies of $A$.  Then $\Ho(\beta)$ fixes those elements of $\HH_{d}(A)$ corresponding to self-dual morphisms.  The Hochschild homology classes that correspond to isomorphisms are important enough to have their own name:
\begin{defn}
Let $\eta\in \HH_{d}(A,A)$ be a class corresponding to an isomorphism
\[
A^{\vee}[d]\rightarrow A.
\]
Then $\eta$ is called \textit{nondegenerate}.
\end{defn}
\begin{defn}
\label{plusnot}
Let $A$ be a homologically finite unital differential graded algebra, and let $\eta\in \HH_t(A)$.  Then we denote by 
\[
\eta^+:A^{\vee}[t]\rightarrow A
\]
the morphism associated to $\eta$ under the isomorphism $\HH_t(A)\cong \Ext^{-t}(A^{\vee},A)$.
\end{defn}
We recall (\cite{VdB10} Proposition C.1):
\begin{lem}
\label{trivflip}
The morphism $\beta$ induces the identity morphism on Hochschild homology.
\end{lem}
It follows trivially that every isomorphism $A^{\vee}[d]\rightarrow A$ is automatically self-dual.  We'll give an independent proof of this statement in the case of fundamental group algebras in Section 4.
\medbreak
We now come to the crucial definition, due to Keller.
\begin{defn}
A unital differential graded algebra $A$ is called \textit{exact} Calabi-Yau of dimension $d$ if there is a nondegenerate class $\eta\in \HH_{d}(A)$ which is in the image of the boundary map
\[
\xymatrix{
\HC_{d-1}(A)\ar[r]^B& \HH_d(A).
}
\] 
\end{defn}

\section{Some topological basics}
\subsection{Loop spaces}
We are concerned throughout with compact differentiable manifolds, which are of course topological spaces.  Given a topological space $X$ we denote by $\Ch^*(X)$ the differential graded algebra of singular cochains on $X$.  The algebra structure is induced by the standard cup product.  Dually, we denote by $\Ch_*(X)$ the differential graded vector space of singular chains on $X$.
\smallbreak
Given a space $X$ we define $LX$ to be the space of continuous maps from $\Ss^1$ to $X$.  $\Ss^1$ has a natural basepoint.  If $X$ has a basepoint also, then we define $\Omega X$ to be the subspace of $LX$ consisting of basepoint preserving maps.  By definition, the path components of $\Omega X$ are in bijective correspondence with $\pi_1(X,x)$, the usual fundamental group of $X$.  There is a natural map
\[
\Omega X \times\Omega X \rightarrow \Omega X
\]
given by concatenation of loops.  One can easily see that this does \textit{not} give $\Ch_*(\Omega X)$ the structure of an associative algebra, just by considering the degree zero part.  However it \textit{is} the archetypal $A_{\infty}$-algebra, and so, in particular, we obtain an algebra $\Ho_*(\Omega X)$ (see \cite{keller-intro} for a nice introduction to $A_{\infty}$-algebras, and \cite{KLH} for a more comprehensive guide, though in fact no essential use of $A_{\infty}$-algebras will be made in this paper).  Given a space $X$, a \textit{cover} of $X$ is a continuous map $f:Y\rightarrow X$ such that for any point $x\in X$ there is an open $U$ containing $x$, and an isomorphism 
\[
g:f^{-1}(U)\rightarrow\coprod_{i\in I} U_i
\]
where each $U_i$ is identified with $U$, such that the diagram
\[
\xymatrix{
\coprod_{i\in I} U_i\ar[dr]\\
f^{-1}(U)\ar[u]^g\ar[r]^f & U
}
\]
commutes.  Given a space $X$ satisfying certain connectedness assumptions (all our spaces will indeed satisfy these assumptions), one can show that there is a connected cover $\mathcal{U}X$ of $X$ with trivial fundamental group, which is called the \textit{universal cover}.  This is unique up to homeomorphism.
\smallbreak
Given a topological group $G$ one can always construct a weakly contractible $G$-space $EG$ such that $G$ acts on $EG$ freely.  Taking the quotient by the $G$ action we arrive at the \textit{classifying space} $BG$, which is unique up to homotopy equivalence.  If the universal covering space $\mathcal{U}X$ of a space $X$ is contractible, then one can take it to be $E\pi_1(X)$, where $\pi_1(X)$ is given the discrete topology, since $\pi_1(X)$ acts freely on it by deck transformations, and one recovers that $X$ is homotopic to $BG$ for the group $G=\pi_1(X)$.
\begin{defn}
A compact manifold $M$ will be called \textit{acyclic} if its universal cover is contractible.
\end{defn}
One can easily show that for an acyclic manifold $M$ there is a homotopy equivalence $\Omega M\simeq \pi_1(M)$, where $\pi_1(M)$ is considered as a discrete topological space.  Furthermore, this induces a quasi-isomorphism of $A_{\infty}$-algebras
\[
\Ch_{*}(\Omega M)\rightarrow \Ho_0(\Omega M)\rightarrow k[\pi_1(M)],
\]
where $k[\pi_1(M)]$ is considered an algebra concentrated in degree zero, i.e. an $A_{\infty}$-algebra with trivial differential and higher multiplications.  This explains why we can essentially ignore $A_{\infty}$-structures in this paper, since we deal with acyclic manifolds.
\medbreak
Given a topological space $X$, the space $LX$ carries a natural $\Ss^1$-action, given by the action of $\Ss^1$ on itself.  If a space has an action by a group $G$, one defines the equivariant homology and cohomology as follows: by construction, the classifying space $BG$ has on it a principal $G$-bundle, and the total space of this bundle, denoted $EG$, is contractible, and carries a free $G$-action.  Given a $G$-space $X$ we consider instead the space
\[
EG\times X.
\]
This is homotopic to $X$, since $EG$ is contractible.  We give this space the diagonal $G$-action
\[
g(a,b)=(ga,gb).
\]
This action is free, and we let 
\[
EG\times^G X
\]
be the orbit space.  In algebraists' terms, this is something akin to taking the derived functor of $G$-invariant chains, which necessitates taking a `free' resolution of $X$ as a $G$-module.  We define the $G$-equivariant homology and cohomology of $X$, denoted $\Ho_*^G(X)$ and $\Ho^*_G(X)$ respectively, to be simply the ordinary homology and cohomology of $EG\times^G X$.
\smallbreak
\begin{examp}
Consider the space $\Ss^{\infty}$ as the union of an ascending chain of unit balls $\Ss^{2n+1}\subset \mathbb{R}^{2n+2}\cong\Cp^{n+1}$.  Then $\Ss^{\infty}$ can easily be shown to be contractible, and it carries a free $\Ss^1$-action given by complex multiplication.  The quotient space is $\Cp\mathbb{P}^{\infty}\simeq B\Ss^1$.
\end{examp}
We restrict attention now to the case $G=\Ss^1$.  Given a $\Ss^1$-fibration
\[
\Ss^1\rightarrow E\rightarrow B
\]
one obtains the \textit{Gysin long exact sequence}
\begin{equation}
\label{Gysin}
\xymatrix{
\ar[r]&\Ho_n(B)\ar[r]& \Ho_{n-2}(B)\ar[r] &\Ho_{n-1}(E)\ar[r] & \Ho_{n-1}(B)\ar[r]&.
}
\end{equation}
Taking the fibration
\[
\Ss^1\rightarrow \Ss^{\infty}\times X \rightarrow \Ss^{\infty}\times^{\Ss^1} X
\]
we obtain the long exact sequence
\[
\xymatrix{
\ar[r] &\Ho^{\Ss^1}_n(X)\ar[r] &\Ho^{\Ss^1}_{n-2}(X)\ar[r]^B& \Ho_{n-1}(X)\ar[r]& \Ho^{\Ss^1}_{n-1}(X)\ar[r]&,
}
\]
since the projection $\Ss^{\infty}\times X\rightarrow X$ induces an isomorphism in homology.  The existence and exactness of this sequence arise also from the explicit chain models of ordinary and equivariant homology given in \cite{Jones87}.  
\smallbreak
Given a space $X$, $LX$ has a convenient description as a \textit{homotopy fibre product}.  The existence of such objects, and their properties, can be deduced from the existence of a model structure on the category of topological spaces, but there is no need to invoke such technicalities here.  We offer a seemingly ad hoc definition of homotopy fibre products:
\begin{defn}
Given a diagram
\[
\xymatrix{
&A\ar[d]^f\\
B\ar[r]^g&C
}
\]
we form the homotopy fibre product as follows.  We first replace $B$ by $B'$, the space of pairs $(b,\gamma_b)$, where $b$ is a point in $B$ and $\gamma_b$ is a map $[0,1]\rightarrow C$ satisfying $\gamma_b(0)=g(b)$.  Then we replace $g$ by the map taking $(b,\gamma_b)$ to $\gamma_b(1)$.  Finally, we take the fibre product of this new diagram.
\end{defn}
\begin{examp}
\label{freeloopex}
The free loop space $LX$ is the homotopy fibre product of
\[
\xymatrix{
&X\ar[d]_-{\Delta}\\
X\ar[r]^-{\Delta}&X\times X
}
\]
where $\Delta$ is the diagonal embedding.
\end{examp}
We deduce from Example \ref{freeloopex} and general facts about homotopy limits that a homotopy equivalence $X\rightarrow Y$ induces a homotopy equivalence $LX\rightarrow LY$.  We say that $A$ is the \textit{homotopy fibre} of a map $f:B\rightarrow C$ of pointed spaces if it is the homotopy fibre product given by the diagram
\[
\xymatrix{
&\star \ar[d]\\
B\ar[r]^f &C,
}
\]
where $\star$ is the space consisting of a point.  In the above situation we obtain a long exact sequence of homotopy groups
\begin{equation}
\label{hLES}
\xymatrix{
\ar[r]&\pi_{n+1}(C)\ar[r]& \pi_n(A)\ar[r] &\pi_n(B)\ar[r]& \pi_n(C)\ar[r]&\pi_{n-1}(A)\ar[r]&
}
\end{equation}
Here the `kernel' of the map $\pi_1(C)\rightarrow \pi_0(A)$ is the preimage of the connected component containing the basepoint.  
\medbreak
Given a space $X$, endowed with the trivial $\Ss^1$-action, there is a natural $\Ss^1$-equivariant map 
\[
\const:X\rightarrow LX
\]
given by sending a point to the constant loop at that point.  If $X$ is a pointed space, then this is a morphism of pointed spaces, with the point of $LX$ being the constant loop at the basepoint.
\begin{prop}
\label{hequivalence}
Let $X$ be an acyclic space.  Then $\const$ is a homotopy equivalence onto $(LX)_0$, the connected component of $LX$ containing the basepoint, and 
\[
\const^{\Ss^1}_*:\Ho_*^{\Ss^1}(X)\rightarrow \Ho_*^{\Ss^1}((LX)_0)
\]
is an isomorphism.
\end{prop}
\begin{proof}
Let $T$ be the homotopy fibre of the map $\const$.  Then $T$ is naturally identified with the space of basepoint preserving maps $\Ss^2\rightarrow X$, which is naturally identified with the space of basepoint preserving maps $\Ss^1\rightarrow \Omega X$, which is a weakly contractible space since $\Omega X$ is homotopic to a discrete set.  The result follows by using the long exact sequence (\ref{hLES}) to show that $\const$ induces isomorphisms on $\pi_i$ groups, and then applying Whitehead's Theorem.
\smallbreak
Since $\const$ is $\Ss^1$-equivariant, we deduce\footnote{Thanks to Jeff Giansiracusa for pointing out this quick proof} from the diagram of fibre sequences
\[
\xymatrix{
X\ar[d]\ar[r]&X\times^{\Ss^1}\Ss^{\infty}\ar[d]^-{\const^{\Ss^1}}\ar[r]&\Cp\textbf{P}^{\infty}\ar[d]\\
LX_0\ar[r]&LX_0\times^{\Ss^1}\Ss^{\infty}\ar[r]&\Cp\textbf{P}^{\infty}
}
\]
the five lemma, and another application of Whitehead's theorem, that the map $\const^{\Ss^1}$ is a homotopy equivalence.
\end{proof}
\subsection{Simplicial sets}
\label{ssets}
We will be using simplicial sets in this paper.  Excellent references for these are \cite {May92}, and \cite{GoeJar99}, though they are also covered in \cite{Lod98}, which has a more algebraic viewpoint.  We give a quick account.  First recall the category $\Delta$.  This has as objects the numbers $\textbf{n}\geq 0$.  Morphisms from $\textbf{n}$ to $\textbf{m}$ are given by weakly order-preserving maps of sets
\[
[0,...,n]\rightarrow [0,...,m],
\]
with composition defined in the obvious way.  A simplicial set is a functor from $\Delta^{\op}$ to $\Set$, the category of Sets.  These form a category $\SSet$, in which the morphisms are natural transformations.  We define the geometric n-simplex to be the space
\[
|\Delta^n|=\{(x_0,...,x_n)\in\mathbb{R}^{n+1}, 0\leq x_i\leq 1, \sum_{0\leq i\leq n} x_i=1\}.
\]
Each vertex of $|\Delta^n|$ has precisely 1 coordinate equal to 1, and so the ordering of the axes induces an ordering of the vertices.  Let $\theta$ be an order-preserving map $[0,...,n]\rightarrow [0,...,m]$.  Then $\theta$ induces a unique affine map $|\theta|:|\Delta^n|\rightarrow |\Delta^m|$ sending the $i$th vertex to the $\theta(i)$th vertex.
\smallbreak
The \textit{standard} n-simplex is the simplicial set $\Delta^n:=\Hom_{\Delta}(-,\textbf{n})$.  Clearly this is a contravariant functor to sets, and so it is a simplicial set.  If $S$ is a simplicial set there is a bijection of sets $S_n\cong \Hom_{\SSet}(\Delta^n,S)$.  Given a simplicial set $S$ we define its simplex category $S_{\simp}$ to be the category having as objects the maps
\[
\Delta^n\rightarrow S,
\]
and as morphisms the commutative diagrams
\[
\xymatrix{
\Delta^n\ar[dr]\ar[rr]^{\theta}&&\Delta^m\ar[dl]\\
&S.
}
\]
Note that in the above diagram, $\theta$ can be identified with a morphism $\textbf{n}\rightarrow \textbf{m}$ of $\Delta$, by the Yoneda lemma.  Given a simplicial set $S$ we define a functor $F_S$ from $S_{\simp}$ to $\Top$, the category of topological spaces, sending maps $\Delta^n\rightarrow S$ to $|\Delta^n|$, and sending objects of the simplex category of $S$ as given above to $|\theta|$.  Finally, the geometric realisation of $S$, denoted $|S|$, is defined to be the colimit of this functor.

\begin{examp}
\label{nerve}
Let $\overrightarrow{n}$ denote the free path category of the quiver
\[
\xymatrix{
0\ar[r]&1\ar[r]&...&\ar[r]&n.
}
\]
Then the functors $\overrightarrow{n}\rightarrow\overrightarrow{m}$ are in natural bijection with the morphisms $\textbf{n}\rightarrow \textbf{m}$ in $\Delta$.  Given a category $\mathcal{C}$, this enables us to build a simplicial set $B(\mathcal{C})$, such that $B(\mathcal{C})_n:=\Fun(\overrightarrow{n},\mathcal{C})$.  If we view a discrete group $G$ as a category with one object, and one morphism for each element of $G$, this gives us the \textit{nerve} of $G$, denoted $B(G)$.  Taking geometric realizations, we have that $|B(G)|\simeq BG$, hence the notation.
\end{examp}

\bigbreak
Let $S$ and $T$ be simplicial sets.  Then we obtain a contravariant functor $S\times_0 T$ from $\Delta$ to $\Set\times \Set$.  Composing with the functor $\Set\times \Set\rightarrow \Set$ given by taking Cartesian products, we get a new contravariant functor from $\Delta$ to $\Set$.  We define $S\times T$ to be this functor.  So in particular, $(S\times T)_n:=S_n\times T_n$.
\medbreak
A morphism $\theta:\textbf{n}\rightarrow \textbf{m}$ in $\Delta$ is called a codegeneracy if $n>m$.  If $S$ is a simplicial set, then we say that a simplex $s\in S_n$ is degenerate if it factors through a degeneracy, that is, it is part of a diagram
\[
\xymatrix{
\Delta^n\ar[dr]\ar[rr]^{\theta}&&\Delta^m\ar[dl]\\
&S.
}
\]
where $\theta$ is induced by a codegeneracy, i.e. $n>m$.  We say that a simplicial set is \textit{finite} if it has only finitely many nondegenerate simplices.\smallbreak
\begin{examp}
The second simplest simplicial set is $\star$.  This has precisely one simplex in every dimension, and only the 0-dimensional simplex is nondegenerate.
\end{examp}

\bigbreak
We define the fundamental group of a simplicial set $S$ to be the fundamental group of its realization.  Given a simplicial set $S$ with fundamental group $G$ we obtain a unique (up to a naturally defined notion of homotopy) map $S\rightarrow B(G)$ preserving the fundamental group (this follows from the well known Quillen equivalence between simplicial sets and spaces).  
\smallbreak
$B(G)$ has a natural construction as the base of a simplicial principal $G$-bundle $E(G)$ (considering $G$ as a discrete simplicial group).  We define $\mathcal{U}S$ via the pullback diagram
\[
\xymatrix{
\mathcal{U}S\ar[r]\ar[d]& E(G)\ar[d]\\
S\ar[r] &B(G).
}
\]
See Chapter 5 of \cite{GoeJar99} for more details.  $|\mathcal{U}S|$ is a universal cover of $|S|$; in particular, all spaces obtained as realizations of simplicial sets have universal covers.
\bigbreak

Let $T$ be a simplicial set.  Then we define the simplicial homology $\Ho_{\bullet}(T)$ in the obvious way: define by $\Ch_*(T)$ the graded vector space with basis (in degree $n$) given by the $n$-simplices of $T$, and the differential
\[
d=\sum_{0\leq i\leq n}(-1)^i d_i^*
\]
where $d_i^*$ is the map induced by the morphism $d_i$ in the category $\Delta$, which in turn is defined as the unique order-preserving injective map from $\textbf{n-1}$ to $\textbf{n}$ without $i$ in its image.  Given a simplicial set $S$ we denote by $\Ch_*(S)$ the complex constructed in this way.  The graded vector space given by the degenerate simplices forms a subcomplex, which is in fact a direct summand of $\Ch_*(S)$, with complement the \textit{normalized complex} of $S$.  This is defined explicitly (as a subcomplex $\No_*(S)\subset \Ch_*(S)$) as follows: let 
\begin{equation}
\label{normal}
\No_{n}(S)=\bigcap_{i<n} \Ker(d_i^*),
\end{equation}
then it is an easy check that this in fact forms a subcomplex.\smallbreak
$\No_*(S)$ admits a simpler description than the one above.  First, define $S_{\red}$ to be set of nondegenerate simplices of $S$.  Then let $\overline{\Ch}_*(S)$ be the complex whose degree $n$ part has basis given by $1_{s}$ for $s\in (S_{\red})_n$.  Then define
\begin{align*}
\overline{d}^*_i(1_{s})&=d^*_i(1_{s})\hbox{ if } d^*_i(s)\in S_{\red}\\
&=0\hbox{ otherwise.}
\end{align*}
Equivalently, one may form $\DD_*(S)\subset \Ch_*(S)$, the subcomplex of $\Ch_*(S)$ spanned by degenerate simplices, and let $\overline{\Ch}_*(S)=\Ch_*(S)/\DD_*(S)$.

\bigbreak
For an arbitrary simplicial set $S$ we have 
\[
\Ho_{\bullet}(|S|)\cong \Ho_{\bullet}(\Ch_*(S))\cong \Ho_{\bullet}(\overline{\Ch}_*(S))
\]
where homology on the left is the usual singular homology.  We adopt the notation
\[
\Ho_{\bullet}(S):=\Ho_{\bullet}(\Ch_*(S))\cong \Ho_{\bullet}(\overline{\Ch}_*(S)).
\]
We say that a simplicial set $S$ is contractible if its realization is.  Now take instead the \textit{cochain} complex $\Ch^*_{\cpct}(S)$ with basis given by the dual basis to the obvious basis for $\Ch_*(S)$, and define the \textit{compactly supported} cohomology $\Ho_{\cpct}^{\bullet}(S)$ of a simplicial complex $S$ to be the cohomology of this cochain complex.  Then for an arbitrary simplicial set we have instead an isomorphism
\[
\Ho_{\cpct}^{\bullet}(|S|)\cong \Ho_{\cpct}^{\bullet}(S),
\]
where cohomology on the left is compactly supported cohomology.\smallbreak
\begin{defn}
\label{contraction}
Given a simplicial set $S$, and a subsimplicial set $\gamma:U\rightarrow S$, we define $S_{c(U)}$ to be the pushout
\[
\xymatrix{
U\ar[r]\ar[d]^{\gamma}&\star \ar[d]\\
S\ar[r]&S_{c(U)}.
}
\]
\end{defn}
Informally, this is obtained by contracting $U$.  Note that the realization functor $|\cdot |$ is a left adjoint, and so it preserves this colimit.  This makes the following technical lemma a straightforward verification:
\begin{lem}
\label{contsimp}
Let $S$ be a simplicial set.  Let $\{U_i,i\in I\}$ be a set of disjoint finite contractible subsimplicial sets of $S$.  Let $S'=S_{c(\coprod_{i\in I} U_i)}$.  Then there are natural isomorphisms
\begin{equation}
\Ho_{\bullet}(S)\cong \Ho_{\bullet}(S')
\end{equation}
\begin{equation}
\label{contentious}
\Ho^{\bullet}_{\cpct}(S)\cong \Ho^{\bullet}_{\cpct}(S').
\end{equation}
\end{lem}
Note that the finiteness assumption is required for (\ref{contentious}), since compactly supported cohomology is not a homotopy invariant.
\begin{lem}
\label{fincon}
Let $S$ be a finite simplicial set, and let $T\subset S$.  Then $S_{c(T)}$ is finite.
\end{lem}
This lemma follows from the standard fact that a simplicial set is finite if and only if it is a colimit of a finite diagram of standard simplices.
\begin{examp}
An abstract simplicial complex is a special type of simplicial set.  It is given by data $(V,T)$, where $V$ is a set (called the set of \textit{vertices}), and $T$ is a subset of the power set of $V$, containing all the singleton sets, containing no infinite sets, and closed under taking subsets.  Given such data we next give $V$ a total order.  Then we can construct a simplicial set $S$, such that the nondegenerate elements of $S_n$ are the elements of $T$ of cardinality $n+1$, and the face maps $d_i^*$ (restricted to nondegenerate simplices) are determined by the ordering of $V$.  By abuse of notation, we will call a simplicial set a simplicial complex if it is obtained in this way.  Given a (differentiable) manifold $M$, we can always find a simplicial complex $S$ such that $|S|$ is homeomorphic to $M$.  If $M$ is compact this simplicial complex may be chosen to be a finite simplicial set.
\end{examp}
\medbreak
We end this subsection with a well-known proposition.  We include the proof for later use.
\begin{prop}
\label{loops}
Let $X$ be a $BG$ for some discrete group $G$.  Then 
\[
\HH_d(k[G])\cong \Ho_d(LX).
\]
\end{prop}
\begin{proof}
Given a pair of simplicial sets $Y,Z$ we define the simplicial set $\Map(Y,Z)$ as follows:  First, define $\Map(Y,Z)_n:=\Hom_{\SSet}(Y\times \Delta^n,Z)$.  Then, given a homomorphism $\textbf{n}\rightarrow \textbf{m}$ in $\Delta$, we obtain a morphism $Y\times\Delta^n\rightarrow Y\times\Delta^m$ of simplicial sets.  This makes $\Map(Y,Z)$ into a simplicial set.  
\smallbreak
It is a standard fact that the homology $\Ho_d(LX)$ can be calculated as the homology of the simplicial set $\Map(\Ss^1,B(G))$, where $\Ss^1$ is the simplicial set with one nondegenerate 0-simplex, and 1 nondegenerate 1-simplex (this follows from the fact that $B(G)$ is \textit{fibrant}, see \cite{GoeJar99}).  It follows from the definition of $\Ss^1$ that it has exactly $n+1$ $n$-simplices.  Likewise we deduce that $\Ss^1\times \Delta^n$ has exactly $n+1$ nondegenerate $(n+1)$-simplices, and the set of vertices for all of them is the same.  Now a map of simplicial sets to $B(G)$ is determined by its action on the 1-simplices, and so one can show that a morphism from $\Ss^1\times \Delta^n$ to $B(G)$ is determined by its action on exactly one simplex.
\smallbreak
We have drawn the realization of $\Ss^1\times \Delta^1$ in Figure \ref{s1cross}.  The double arrows indicate that we should identify the left and right edges.  In general these edges are replaced with $n$-simplices, and we obtain a unique morphism by giving a morphism of the $(n+1)$-simplex $A$ to $B(G)$.  This is determined by an element of $G^{n+1}$, since the 1-simplices of $B(G)$ are labelled by $G$.  Precisely, the entire morphism of simplicial sets from $\Ss^1\times \Delta^n$ is determined by its action on the edges that join the $i$th vertex to the $(i+1)$th vertex of $A$, for $i\leq n-1$.  This identifies the underlying graded vector space of $\Ch_*(\Map(\Ss^1,B(G)))$ with the underlying graded vector space of the Hochschild complex for $k[G]$.  It is easy to check that the differentials are identified too.
\end{proof}
\begin{figure}

\centering
\input{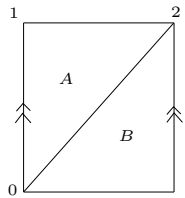}
\caption{The simplicial decomposition of $\Ss^1\times \Delta^1$}
\label{s1cross}
\end{figure}
\begin{rem}
For an acyclic manifold $M$ the space $LM$ breaks into path components labelled by conjugacy classes in $\pi_1(M)$.  The corresponding decomposition of $C(k[\pi_1(M)])$ is given as follows: each summand $k[\pi_1(M)]^{\otimes n}$ in the bar resolution for $k[\pi_1(M)]$ carries a natural $\pi_1(M)$-grading.  This in turn induces a decomposition of the Hochschild complex, labelled by conjugacy classes in $\pi_1(M)$.  From the above proof of Proposition \ref{loops}, we see that this is the same as the decomposition coming from the decomposition of $LM$ into path components.
\end{rem}
\section{Smooth algebras and superpotential algebras}
\subsection{Smooth algebras}
We recall a few notions from noncommutative geometry.  These are by now pretty standard, there are excellent (and comprehensive) references: (\cite{Ginz01}, \cite{Ginz}, \cite{CBEG}...).\smallbreak
Let $A$ be a unital differential graded algebra over $k$.  First we define the differential graded bimodule of 1-forms
\[
\Omega_k^1 A:=\Ker(m:A\otimes A\rightarrow A),
\]
a noncommutative analogue of the global vector fields on a scheme.\smallbreak 
In this section we present global deformed preprojective algebras (superpotential algebras).  These were first introduced in \cite{Ginz}.  In the formal case, constructions begin with a formal completion of a quiver algebra.  We think of this as a formal neighborhood in a noncommutative affine smooth scheme.  It is natural, then, in attempting to globalise the picture, to start with a `noncommutative smooth scheme', following the framework of \cite{KontRos00}, \cite{CBEG}, \cite{Ginz}.  
\begin{defn}[\cite{CunQui95}]
A finitely generated graded algebra $A$ is \textit{smooth} if the $A$-bimodule $\Omega^1_k A$ is projective.
\end{defn}
One of the central features of smooth algebras is that they satisfy a lifting property for nilpotent extensions.  This corresponds to smoothness under the representation functor of \cite{CBEG}, \cite{Ginz05}.
\begin{examp}
Let $A$ be the free path algebra of a quiver.  Then $A$ is smooth in the above sense.
\end{examp}
\begin{examp}
We can also form smooth algebras by localization.  For example, let $Q$ be a quiver, then we can localize $k[Q]$, the free path algebra, by inverting all the arrows.  The resulting algebra is smooth.  It follows that the fundamental group algebra of $\Ss^1$ is smooth.
\end{examp}
\begin{prop}
Let $A$ be an augmented differential graded algebra, complete at its augmentation ideal $\mathfrak{m}$.  Then $A$ is smooth if and only if the natural map
\[
\hat{T}(\mathfrak{m}/\mathfrak{m}^2)\rightarrow A
\]
from the completed free tensor algebra is an isomorphism.
\end{prop}
\begin{prop}
\label{glofor}
Let $A$ be a smooth augmented differential graded algebra.  Then $\hat{A}$, the algebra obtained by completing at the augmentation ideal, is smooth.
\end{prop}
\smallbreak
\begin{examp}
\label{aff}
Let $A$ be a smooth algebra over $k$, and let $V$ be a finitely generated negatively graded $A$-bimodule.  Let $T_A(V)$ be the free unital associatative algebra generated by $V$.  Then we consider $T_A(V)$ as the noncommutative affinization of a sheaf on a noncommutative affine scheme.  In the case in which $V$ is free as a bimodule, we will call this a noncommutative vector bundle over $A$.
\end{examp}
\begin{prop}
Let $A$ be a smooth algebra, and let $V$ be a finitely generated free $V$-bimodule.  Then $T_A(V)$ is smooth.
\end{prop}
Given a complex of $A$-bimodules, we obtain a unique differential on the associated affinization, acting by derivations.  This is the form our noncommutative dg-schemes will take.
\begin{defn}
Let $(A,d)$ be a pair of a graded smooth unital differential graded algebra $A$ and a degree 1 differential $d$ on $A$, acting by derivations.  Then we call $(A,d)$ a \textit{semismooth} differential graded algebra if the underlying algebra of $A$ is a nonpositively graded vector bundle over a smooth algebra concentrated in degree zero.
\end{defn}
\begin{rem}
This turns out to be a weak requirement.  Indeed, as long as there is a finite free resolution of the diagonal bimodule of an ordinary algebra $B$, we can find a semismooth presentation, i.e. a semismooth $A$ and a quasi-isomorphism
\[
A\rightarrow B.
\]
\end{rem}
The principal reason for restricting to such presentations is that they give a good handle on the mixed complex $C(A)$.  To explain this, we introduce a different mixed complex $M(X(A))$ from \cite{CunQui95b}.  Let
\[
\lambda: \Omega_k^1 A\rightarrow A\otimes A
\]
be the natural bimodule map sending $Da$ to $a\otimes 1-1\otimes a$ (this is just the inclusion).  Then applying $A\otimes_{A^e} -$ to this map we obtain a two-term complex
\[
\sigma:(\Omega_k^1)_A\rightarrow A.
\]
We define a $B$-map going the other way by $B(a)=\overline{(a\otimes 1-1\otimes a)}$.  
\begin{prop}
\label{mixquas}
Let $A$ be a semismooth algebra.  Then the natural morphism of mixed complexes
\[
\theta: C(A)\rightarrow M(X(A))
\]
is a quasi-isomorphism.
\end{prop}
The proof is as in \cite{VdB10}.  The case in which $A$ has zero differential follows via the definition of Hochschild homology, since
\[
\xymatrix{
\Omega_k^1 A\ar[r]&A\otimes A\ar[r]& A
}
\]
is a projective bimodule resolution of $A$.  Next one uses the filtration arising from the construction of the underlying algebra of $A$ as a tensor algebra, and  the (third quadrant) spectral sequence for filtered complexes.
\subsection{Noncommutative geometry}
Let $A$ be an arbitrary smooth differential graded algebra over $k$.  There is a canonical derivation 
\[
D:A\rightarrow \Omega_k^1 A
\]
given by $D(a)=a\otimes 1-1\otimes a$.  Given a differential graded algebra $A$ we define $\Der_k(A,M)$ to be the differential graded vector space of $k$-linear \textit{super}derivations from $A$ to $M$, and we define $\Der_k(A):=\Der_k(A,A)$.  The bimodule $\Omega_k^1 A$ has the universal property that there is a natural isomorphism
\[
\Der_k(A,M)\rightarrow \Hom_{A^e}(\Omega_k^1 A,M)
\]
taking $\vartheta\in \Der_k(A,M)$ to the homomorphism taking $aDb$ to $a\vartheta(b)$.  Next, define the differential graded algebra of noncommutative differential forms
\[
(\Omega^{\bullet}_k A,D)
\]
where $\Omega^{\bullet}_k A:=T_{A}(\Omega_k^1 A)$, and $D$ is the \textit{de Rham} differential above, which induces a unique differential on $\Omega^{\bullet}_k A$.  Next we define
\[
\DR_k(A):=T_A(\Omega_k^1 A)_{\cyc}.
\]
We define
\[
\DDer(A)=\Der_k(A,A\otimes A).
\]
There is a natural isomorphism $\DDer(A)\cong\Omega_l^1 A^{\vee}$, and so in particular $\DDer(A)$ is a $A$-bimodule.  This is called the bimodule of \textit{double derivations} of $A$.\smallbreak
Given a $\vartheta\in \Der_k(A)$, we obtain a derivation (\textit{Lie derivative}) $L_{\vartheta}$ on $\Omega_k^{\bullet} A$, defined on generators by
\begin{eqnarray*}
&L_{\vartheta}:a\rightarrow \vartheta(a)\\
&L_{\vartheta}:Da\rightarrow D(\vartheta(a)).
\end{eqnarray*}
We also have the contraction mapping $i_{\vartheta}$ defined by
\begin{eqnarray*}
&i_{\vartheta}:a\rightarrow 0\\
&i_{\vartheta}:Db\rightarrow \vartheta(b).
\end{eqnarray*}
Both of these maps descend to maps on the quotient $\DR_k(A)$.  If $\lambda\in\DDer_k(A)$ is a double derivation, we mimic the definition of the usual contraction mapping to obtain a double derivation on $(\Omega_k^{\bullet} A)$.  Next we define the \textit{reduced} contraction mapping
\[
\iota_{\lambda}:(\Omega_k^{\bullet} A)\rightarrow (\Omega_k^{\bullet} A)
\]
by setting 
\[
\iota_{\lambda}=m\beta(i_{\lambda}),
\]
where $m$ is the multiplication on $(\Omega_k^{\bullet} A)$ and $\beta$ is the swap morphism on $(\Omega_k^{\bullet} A)^{\otimes_k 2}$ (note that a sign appears here due to the usual Koszul sign rule).\smallbreak
We will be interested in a noncommutative version of \textit{symplectic} 2-forms.
\begin{defn}
Let $\omega\in \DR^2_k A$ be a \textit{closed} 2-form.  Then $\omega$ determines a map
\[
i^{\omega}:\Der_k(A)\rightarrow \DR^1_k(A)
\]
given by
\[
i^{\omega}(\theta)=i_{\theta}(\omega).
\] 
We say that $\omega$ is \textit{symplectic} if this map is an isomorphism.  Analogously, $\omega$ determines a map
\[
\iota^{\omega}:\DDer(A)\rightarrow \Omega_k^1 A,
\]
and we say that $\omega$ is \textit{bisymplectic} if this is an isomorphism.
\end{defn}
For a smooth ring, a 2-form is bisymplectic only if it is symplectic (\cite{CBEG}).
\begin{defn}
Given a symplectic or bisymplectic 2-form $\omega$, we define $\Der_{k,\omega}(A)$, the space of symplectic derivations, to be the subset of $\Der_k(A)$ consisting of 1-forms $\vartheta$ satisfying $L_{\vartheta}\omega=0$.
\end{defn}
\begin{prop}
\label{caris}
Let $\omega$ be a symplectic 2-form.  Then $i^{\omega}$ restricts to an isomorphism 
\[
\Der_{k,\omega}(A)\rightarrow \Ker(D:\DR^1_k(A)\rightarrow\DR^2_k(A)).
\] 
\end{prop}
Let $\omega\in \DR_k^2(A)$ be bisymplectic.  Then we define, as in \cite{CBEG},
\[
H_a:=(\iota^{\omega})^{-1}(Da)
\]
and
\[
\{a,b\}_{\omega}=m(H_a(b)).
\]
\subsection{Ginzburg's dg algebra}
We first remind the reader of the construction of the \textit{noncommutative moment map} from \cite{CBEG}.\smallbreak
We define the canonical double derivation $\Delta\in \DDer(A)$ by
\[
\Delta(a)=a\otimes 1-1\otimes a.
\]
Then (see \cite{CBEG}) there is a function
\[
\ncm:(\DR^2_k A)_{\mathrm{closed}}\rightarrow A/k
\]
satisfying the property that $D(\ncm(\omega))=\iota_{\Delta}\omega$.  Then we have the following proposition:
\begin{prop}[\cite{CBEG} Proposition 4.1.3]
\label{CBEGlem}
Let $\textbf{w}\in A$ be a representative of $\ncm(\omega)$.  Then $D\textbf{w}=0$ in $\DR^1_k(A)$.
\end{prop}
\begin{defn}
\label{conndef}
Let $A$ be a differential graded algebra.  Then we say $A$ is \textit{connected} if the following sequence is exact
\[
\xymatrix{
0\ar[r]& k\ar[r]& \DR_k^0 (A)\ar[r]&\DR_k^1(A).
}
\]
\end{defn}
Note that there is an identification $\DR_k^0(A)=A_{\cyc}$.  Let $A$ be a connected algebra.  
We deduce that the $\textbf{w}$ of Proposition \ref{CBEGlem} belongs to the vector space generated by $[A,A]$ and $k\subset A$.
\begin{rem}
This gives rise to a subtlety in 2 dimensions.  If $\omega$ has degree greater than zero, then there is a unique homogeneous representative $\textbf{w}$ of $\ncm(\omega)$, of the same degree as $\omega$, and furthermore it is an element of $[A,A]$.  If the degree of $\omega$ is zero (which will correspond to the 2-dimensional case) then this helpful fact no longer obtains.  This gives rise to two different definitions of a superpotential algebra in 2 dimensions, we choose the stronger one.  Note that if we do not demand that $\textbf{w}\in [A,A]$ then the proof of Theorem \ref{thm1} fails.

\end{rem}

\begin{prop}\cite{CBEG}
Let $A$ be a connected algebra.  Then there is a unique function
\[
\ncml:(\DR^2_k A)_{\mathrm{closed}}\rightarrow [A,A]
\]
such that, composing with the map $[A,A]\hookrightarrow A \twoheadrightarrow A/k$ we obtain $\ncm$.
\end{prop}
\begin{rem}
Again we see that this is rather trivial in case the degree of $\omega$ is nonzero -- just pick $\ncml(\omega)$ to be the homogeneous part of $\ncm(\omega)$ that has nonzero degree.
\end{rem}
We come now to the main definition/construction of this section.  Let $(A,\omega, \xi)$ be a triple, where $A$ is a connected nonpositively graded noncommutative vector bundle over a smooth algebra concentrated in degree zero, $\omega$ is a bisymplectic 2-form, homogeneous with respect to the induced grading from $A$, and $\xi\in\Der(A)$ satisfies $L_{\xi}\omega=0$ and $\xi^2=0$, and has cohomological degree 1.  Then following \cite{Ginz} we define $\mathfrak{D}(\omega,\xi)$, the \textit{GDGA} associated to this data, as follows: we let the underlying algebra of $\mathfrak{D}(\omega,\xi)$ be $A\ast_k k[t]$, where $t$ is placed in degree $c-1$, and $c$ is the degree of $\omega$ with respect to the grading induced by $A$.  We produce a differential on $A$ by setting 
\[
da=\xi(a).
\]
We extend this to a derivation on $\mathfrak{D}(\omega,\xi)$ by setting
\[
dt=\ncml(\omega).
\]
Our conditions on $\xi$ guarantee that $d^2=0$.
\begin{defn}
We say that an algebra $B$ is a \textit{superpotential algebra} if it is quasi-isomorphic to a semismooth algebra $A$ given by the above construction, for $\xi=\{W,\hbox{ }\}$, with $W\in A_{\cyc}$.  This is strictly stronger than the notion of GDGA just defined.
\end{defn}
\begin{examp}
\label{spheres}
Let $k=\mathbb{Q}$, and let $A=k$.  Clearly $A$ is a connected smooth algebra.  Then we can consider the 2-form $\omega=0$ as an element of the $(2-n)$th (cohomologically) graded piece of $\DR_k^2(A)$.  It is closed, and $\DDer(A)=\Omega_k^1(A)=0$, so we can apply the above construction, with $W=0$.  We obtain $k\langle x[1-n]\rangle$, which is quasi-isomorphic (as an $A_{\infty}$-algebra) to $\Ch_*(\Omega \Ss^{n})$.  This follows from the formality of $\Ch^*(\Ss^{n})$, and usual Koszul duality between symmetric and exterior algebras.  In this paper we will only be interested in superpotential algebras concentrated in degree zero, but one may just as well consider more general differential graded algebras as superpotential algebras, if they are given by the same construction; this example shows that $\Ch_*(\Omega \Ss^{n})$ is a superpotential algebra in this sense.  It is straightforward to show directly that $k\langle x[1-n] \rangle$ is a Calabi-Yau algebra of dimension $n$.
\end{examp}
\begin{examp}
\label{3-torus}
Let $k=\mathbb{Q}$ again, and let 
\[
A=\mathbb{Q}\langle x^{\pm 1}, y^{\pm 1}, z^{\pm 1}, x^*[1],y^*[1],z^*[1]\rangle.
\]
This algebra is smooth, as it is the localization of a free quiver algebra, and it is connected.  We let 
\[
\omega=\frac{1}{2}(DxDx^*+DyDy^*+DzDz^*).
\]
Finally, let $W=xyz-xzy$.  Then the resulting superpotential algebra is quasi-isomorphic to $\mathbb{Q}[\pi_1((\Ss^1)^3)]$, the fundamental group algebra of a 3-torus.
\end{examp}
\begin{examp}
Let $(A,d)$ be a superpotential algebra, with $A$ formed by taking the total space of a noncommutative vector bundle $V$ over a smooth algebra $B$.  Let $\mathfrak{m}$ be a two-sided ideal of $B$, then we let $\hat{B}$ denote the completion of $B$ along $\mathfrak{m}$.  By tensoring $V$ with the completion we obtain a new vector bundle over $\hat{B}$, and indeed a new superpotential description at the formal neighborhood of $\mathfrak{m}$.  In the 3-dimensional case, in which the degree of $\omega$ is -1, we obtain, under the representation functor, a new description of the formal neighborhood defined by $\mathfrak{m}$ as a critical locus, which is the same as the restriction of the global description of the representation stack of the superpotential algebra, as a critical locus, to the formal neighborhood.
\end{examp}
We will be concerned with counterexamples to the following conjecture, which is weaker than the analogous Conjecture 6.2.1 in \cite{Ginz}, when restricted to 3 dimensions.
\begin{conj}
\label{mfldconj}
If $M$ is an acyclic oriented compact manifold of dimension $d$, then $k[\pi_1(M)]$ is a superpotential algebra.
\end{conj}
We next consider a weakening of the construction.  Let $A$ be now a nonpositively graded noncommutative vector bundle over a smooth connected algebra concentrated in degree zero, and let $\xi$ be a derivation of $A$ satsifying $\xi^2=0$.  Following \cite{Ginz}, \cite{CBEG} we say that a closed 2-form $\omega\in \DR^2_k(A)$ is \textit{homologically symplectic} if the map
\[
\iota_{\omega}:\DDer_k(A)\rightarrow \Omega_k^1 A
\]
is a quasi-isomorphism with respect to the differentials $L_{\xi}$ (we assume also that $L_{\xi}\omega=0$).  Given such data we proceed as before to build an algebra $\mathfrak{D}(\omega,\xi)$.
\begin{defn}
An algebra that is quasi-isomorphic to one given by this weaker construction will be called a homological GDGA.
\end{defn}
We come now to the main theorems regarding GDGAs.
\begin{thm}\cite{Ginz}
\label{Gthm}
Let $B$ be an algebra such that there exists a triple $(A,\omega,\xi)$ giving rise to the GDGA $\mathfrak{D}(\omega,\xi)$, and a quasi-isomorphism $\mathfrak{D}(\omega,\xi)\rightarrow B$.  Let $\omega$ have cohomological degree $-n$. Then $B$ is a Calabi-Yau algebra of dimension $n+2$.
\end{thm}
\begin{thm}
\label{thm1}
Let $B$ be as above.  Then in fact $B$ is exact Calabi-Yau of dimension $n+2$.
\end{thm}
\begin{proof}
We may calculate the cyclic and Hochschild homologies of $B$ by considering $\mathfrak{D}:=\mathfrak{D}(\omega,W)$, which is semismooth by definition.  It follows that we may calculate these homologies, and their connecting maps, via the mixed complex $M(X(\mathfrak{D}))$.
\smallbreak
By construction, $dt\in [\mathfrak{D},\mathfrak{D}]$.  It follows that there is some $\epsilon\in (\Omega_k^1\mathfrak{D})_\mathfrak{D}$ such that $\sigma\epsilon=dt$, and so $(\epsilon,t)$ is a class in $\HC_{n+1}(\mathfrak{D})$.  Under the map 
\[
B:M(X(\mathfrak{D}))\rightarrow M(X(\mathfrak{D})),
\]
$(\epsilon,t)$ is sent to $(Dt,0)$.  This is just the nondegenerate class appearing in the proof in \cite{Ginz} that $\mathfrak{D}$ is Calabi-Yau.
\end{proof}
\begin{rem}
\label{con2}
The second part of Theorem 3.6.4 of \cite{Ginz} concerns a partial inverse to Theorem \ref{Gthm}.  Namely, the result states that if $A$ satisfies also the condition that kernels of maps between finitely generated free bimodules are finitely generated (this is called the \textit{friendliness} condition), then the Calabi-Yau condition implies that we may find a homological GDGA algebra presentation for it.  A sketch proof is provided in dimensions 3 and higher.  It is conjectured in \cite{Ginz} that we can improve on this, and find a GDGA description (at least in dimension greater than 2).  The present paper suggests that this conjecture is false, with counterexamples coming from fundamental group algebras of acyclic manifolds, but there are some technical issues that one must face here.  The principal difficulty is to show that our algebras are friendly - this appears to be yet another difficult problem regarding group algebras.
\end{rem}
\section{Fundamental group algebras of manifolds}
\subsection{Formal neighborhood of the trivial module}
This subsection is not strictly necessary to the logical flow of the paper.  We include it for two reasons.  Firstly, to give a little feel for the noncommutative geometry of fundamental group algebras, and secondly, to make clear the contact with \cite{VdB10} and \cite{Seg08}.  Given the status of this subsection in the paper, some terms will be introduced that are not defined here, but are in \cite{VdB10}.  We consider $A:=k[\pi_1(M)]$, for $M$ an acyclic compact oriented manifold, as a noncommutative affine scheme.  We are first of all interested in the formal behaviour of this noncommutative scheme at the augmentation ideal $\mathfrak{m}$.  Precisely, we define
\[
\hat{A}=\underleftarrow{\lim}(A/A\mathfrak{m}^{n}A).
\]
We consider this as a pseudocompact algebra, and call this completion the formal scheme at the distinguished point, since that is precisely what it is in the case in which $M=(\Ss^1)^d$ and we are actually doing toric geometry, in the sense of \cite{ToricVarieties}.  Under the representation functor 
\[
B\rightarrow \Hom(B,M_{n\times n}(k))
\]
this gives us the subspace given by the formal neighborhood of unipotent representations of $A$.  The following theorem is (at least in some form or other) standard.
\begin{thm}
The Koszul dual $B$ of $\hat{A}$ is, up to quasi-isomorphism, $\Ch^*(M)$, the differential graded algebra of cochains on the manifold $M$ with the usual cup product.
\end{thm}
\begin{proof}
As a graded algebra, $B$ is given by
\[
B=k\oplus \prod_{i\geq 1} (\mathfrak{m}^*)^{\otimes i}[i],
\]
with multiplication given by the usual multiplication in tensor algebras.  Let $S$ be the \textit{nerve} of $\pi_1(M)$.  This is a simplicial set satisfying $|S|\simeq M$ (see Example \ref{nerve}).  Next note that the algebra structure on $\Ch^*(M)$ naturally descends to an algebra structure on $\Ch^*(M)/\DD^*(M)$.  We establish a map of graded algebras
\[
B\rightarrow \Ch^{*}(M)/\DD^*(M)
\]
by sending $(1-g_1)^*\otimes...\otimes (1-g_i)^*$ to $(g_1,...,g_i)^*$.  One easily checks that this is in fact an isomorphism of differential graded algebras.
\end{proof}
Now we are firmly in the situation of \cite{VdB10}, and so we quickly deduce the following propositions:
\begin{prop}
The pseudocompact algebra $\hat{A}$ is Calabi-Yau, of dimension $\dim(M)$.
\end{prop}
Homological finiteness will follow from the (global) homological finiteness of $A$, proved later.  Once this is given, the proposition follows directly from the Koszul duality statement, and the existence of a pairing on cohomology of $M$ (Poincar\'{e} duality).
\begin{prop}
There is a natural quasi-isomorphism $k\otimes_{\hat{A}} k\simeq \Ho_{*}(M)$.  There is also a natural map $\hat{A}\otimes_{\hat{A}^e}\hat{A}\rightarrow k\otimes_{\hat{A}} k$.  A class in $\HH_{\dim(M)}(\hat{A})$ is nondegenerate if and only if its image in $\Ho_{*}(M)$ is nonzero, or, equivalently, if it is a nonzero multiple of the fundamental class.
\end{prop}
\begin{proof}
Let $\Ho_{\infty}^*(M)$ be a minimal model for $\Ch^{*}(M)$.  Then $\hat{A}$ is quasi-isomorphic to the Koszul dual of $\Ho_{\infty}^*(M)$.  We now apply \cite{VdB10} Lemma 11.1.2.
\end{proof}
\begin{prop}
The pseudocompact algebra $\hat{A}$ is exact Calabi-Yau.
\end{prop}
\begin{proof}
This is the main result of \cite{VdB10}.
\end{proof}
\begin{rem}
It is interesting that one can always find, formally, a superpotential description, even though it may not be possible (as we shall see) to extend this description globally.  This is a familiar issue connected to the Donaldson-Thomas theory of objects in geometric 3-Calabi-Yau categories (see e.g. the discussion in \cite{Behr09}).  For example in \cite{PT12} a scheme is given, with a symmetric perfect obstruction theory, that does not admit a critical locus description (though this scheme is not constructed as a moduli space of objects in a 3-Calabi-Yau category).  It would be interesting to know whether the moduli spaces of framed representations of fundamental group algebras of hyperbolic 3-folds provide further examples of this phenomenon, arising from genuine moduli spaces of objects in 3-Calabi-Yau categories.
\end{rem}
\subsection{The global geometry of fundamental group algebras}
We now turn our attention from the formal geometry around the distinguished point to the global situation.  We start with a propositon regarding the differential geometry of $k[\pi_1(M)]$.
\begin{prop}
Let $G$ be any group.  Then $k[G]$ is connected in the sense of Definition \ref{conndef}.
\end{prop}
\begin{proof}
We first show that $k[G]\neq [k[G],k[G]]$.  To see this, consider the action of a commutator $xy-yx$ on the trivial module $k$.  Obviously $xy-yx$ annihilates $k$, from which we deduce that $xy-yx$ is in the augmentation ideal for $k[G]$.\smallbreak
Next we must show that $\Ker(k[G]\rightarrow \DR_k^1 k[G])$ is $[k[G],k[G]]+k$.  The vector space $k[G]$ has a natural basis $\{g|g\in G\}$, and so this gives us a natural spanning set for $[k[G],k[G]]$, and for $[\Omega_k^1 k[G], k[G]]$.  With respect to this spanning set the assertion is easy to verify.
\end{proof}
The following theorem is due to Kontsevich, though no proof appears in the literature, so it is provided here.
\begin{thm}
\label{hfinite}
Let $M$ be a compact $d$-dimensional acyclic manifold.  Then $A:=k[\pi_1(M)]$ is homologically finite.
\end{thm}
\begin{proof}
We construct a bimodule resolution of $A$.  Let $S$ be a finite simplicial complex such that there is a homeomorphism $|S|\simeq M$.  Let $T$ be a maximal tree in $S$.  Then we replace $S$ by $S_{c(T)}$ (see Definition \ref{contraction}).  Note that by Lemma \ref{contsimp} this makes no difference to the homology or compactly supported cohomology (which in this case just is the cohomology).  So we may assume (using Lemma \ref{fincon}) that $S$ has only finitely many nondegenerate simplices, and only one $0$-simplex.  
\smallbreak
Let $\Gamma(S)$ have as underlying graded $A$-bimodule
\[
\Gamma(S):=A\otimes \No_*(S)\otimes A,
\]
where $\No_*(S)$ is the normalized complex for $S$ (see (\ref{normal}), Section \ref{ssets}).  \smallbreak
Then $\Gamma(S)$ is a finitely generated free $A$-bimodule.  It is, furthermore, a sub-bimodule of
\[
\tilde{\Gamma}(S):=A\otimes \Ch_*(S) \otimes A.
\]
Note that we are, so far, considering these objects just as graded bimodules, so we are forgetting the differential that already exists on $\Ch_*(S)$.
\smallbreak
We next describe the differential on $\tilde{\Gamma}(S)$.  We define maps $\zeta_i,\gamma:\mathbf{1}\rightarrow \mathbf{n}$ by
\[
\zeta_i(0)=i,\hbox{  }\zeta_i(1)=1+i,
\]
and
\[
\gamma(0)=0,\hbox{  }\gamma(1)=n.
\]
Given a simplex $s$ of $S$, $\zeta_i(s)$ and $\gamma(s)$ are naturally elements of $\pi_1(M)$.  We write $1_s$ for the element of the basis of the graded vector space $\Ch_*(S)$ corresponding to $s$.  Let $s$ be a $t$-simplex of $S$.  Then define
\begin{eqnarray*}
&d(1_{s})=\zeta_0(s)\otimes 1_{d^*_0(s)}\otimes 1 +\\& \sum_{1\leq i\leq t-1}(-1)^i \otimes 1_{d^*_i(s)}\otimes 1 +(-1)^{t}\otimes 1_{d^*_t(s)}\otimes \zeta_{t-1}(s).
\end{eqnarray*}
We give the above complex a $\pi_1(M)$-grading by taking the degree of $1_s$ to be $\gamma(s)$.  Note that our differential preserves this grading.  One easily checks that this differential makes $\tilde{\Gamma}(S)$ a complex of $A$-bimodules, with subcomplex $\Gamma(S)$ (this becomes clear after restricting to a particular $\pi_1(M)$-grade).  Fix some $g\in \pi_1(M)$.  Consider the universal cover
\begin{equation}
\label{simplcover}
\mathcal{U}S\rightarrow S
\end{equation}
of $S$.  By picking an arbitrary $0$-simplex of $\mathcal{U}S$ we make (\ref{simplcover}) into a morphism of pointed simplicial sets.  Note that by Lemma \ref{contsimp}, $\mathcal{U}S$ is weakly contractible, since it is obtained by contracting a forest of finite trees in the universal cover of our original $S$.  We next construct a map, for any $g\in \pi_1(M)$,
\[
\mu:\tilde{\Gamma}(S)_{(g)}\rightarrow \Ch_{*}(\mathcal{U}S),
\]
where $\tilde{\Gamma}(S)_{(g)}$ is the $g$-graded piece of our bimodule $\tilde{\Gamma}(S)$.  We define
\[
\mu(g'\otimes 1_{s}\otimes g'')=g'(s)
\]
where we identify $g'$ with the deck transformation
\[
g':\mathcal{U}S\rightarrow \mathcal{U}S
\]
it induces, and identify $s$ with the unique lift to the universal cover that has the basepoint at its $0$th vertex.  Then it is easy to see that in fact $\mu$ is an isomorphism of chain complexes.  Furthermore, it induces an isomorphism of subcomplexes
\[
\Gamma(S)_{(g)}\rightarrow N_*(\mathcal{U}S)_{(g)}.
\]
and so, in particular, the inclusion 
\[
\Gamma(S)_{(g)}\rightarrow \tilde{\Gamma}(S)_{(g)}
\]
is a quasi-isomorphism.  Finally, the map
\[
\tilde{\Gamma}(S)\rightarrow A
\]
defined by
\[
g'\otimes 1_{0}\otimes g''\rightarrow g'g''
\]
is a quasi-isomorphism, since the natural map
\[
\Ch_{*}(\mathcal{U}S)\rightarrow k
\]
is, and we have proved homological finiteness.
\end{proof}

From Theorem \ref{hfinite} we deduce that
\[
\Hom_{A^e}(A^{\vee}[d],A)\cong \HH_d(A).
\]
\medbreak
Just as the normalized complex of a simplicial set can be considered as a direct summand of $\Ch_*(S)$, or, up to isomorphism, as a complex with underlying basis given by the nondegenerate simplices, $\Gamma(S)$ admits a different description, that is in effect easier to use.  This is given by
\begin{equation}
\overline{\Gamma}(S):=\bigoplus_{s\in S_{\red}} A \otimes k[\dim(s)] \otimes A
\end{equation}
and differential 
\begin{eqnarray*}
&d(1_{s})=\zeta_0(s)\otimes \overline{1_{d^*_0(s)}}\otimes 1 +\\& \sum_{1\leq i\leq t-1}(-1)^i \otimes \overline{1_{d^*_i(s)}}\otimes 1 +(-1)^{t-1}\otimes \overline{1_{d^*_t(s)}}\otimes \zeta_{t-1}(s),
\end{eqnarray*}
where $\overline{1_{d^*_s}}$ is the image of the natural projection from the vector space spanned by all simplices to the vector space spanned by the nondegenerate simplices.  One then proceeds as before, this time considering an isomorphism
\[
\overline{\mu}:\overline{\Gamma}(S)_{(g)}\rightarrow \overline{\Ch}_{*}(\mathcal{U}S).
\]
\medbreak

For the remainder of this Section we will set $A:=k[\pi_1(M)]$, for $M$ a compact acyclic \textit{orientable} manifold.
\smallbreak
We next investigate the question of when a class in $\HH_d(A)$ is nondegenerate.  We first recall
\begin{thm}[\cite{Jones87}]
\label{Jonloop}
There is a quasi-isomorphism of differential graded vector spaces $C(A)\simeq \Ch_{*}(LM)$.
\end{thm}
\begin{proof}
This is Theorem 6.2 of \cite{Jones87}, a slightly different proof appears as Proposition \ref{loops}.  For future reference, though, we write down the map from the complex computing Hochschild homology given by taking the bimodule resolution of $A$ that we have constructed to the usual reduced bar resolution, so that we can actually write down the chains in $LM$ corresponding to chains in our given chain complex.
\smallbreak
It is enough, then, to construct a quasi-isomorphism from $\tilde{\Gamma}(M)\otimes_{A^e} A$ to $C(A)$.  This we do as follows: there is a natural isomorphism
\[
(A\otimes 1_{0}\otimes A)\otimes_{A^e} A \cong A.
\]
To deal with the other components $(A\otimes 1_s\otimes A)\otimes_{A^e} A$ appearing in the complex $\tilde{\Gamma}(M)\otimes_{A^e} A$, it is enough to pick an element of $A^{\otimes i}$ for each $i$-dimensional simplex $s$.  There is a homotopy equivalence $f:S\rightarrow B(G)$, defined uniquely up to homotopy.  Now such a map is determined uniquely by what it does to the nondegenerate simplices of $S$, and a map of simplicial sets to $B(G)$ is determined uniquely by where it sends the $1$-simplices of the preimage.  We deduce that in fact $f$ is uniquely defined, since it must send 1-simplices of $S$ to the 1-simplices of $B(G)$ they are labelled by.  So we use the assignment
\[
1_{s}\rightarrow \zeta_0(s)\otimes...\otimes \zeta_{i-1}(s).
\]
This is a quasi-isomorphism since the inclusion $S\rightarrow B(\pi_1(M))$ is a homotopy equivalence.  It follows that we have constructed a chain model for $LM$.
\end{proof}
\begin{prop}
\label{flip}
The flip map $\beta:A \otimes_{A^e} A\rightarrow A\otimes_{A^e} A$ induces an isomorphism in Hochschild homology.
\end{prop}
\begin{proof}
This follows from Michel Van den Bergh's observation (Lemma \ref{trivflip}).  We include it because the statement in this special case has a nice geometric meaning.  It amounts to the statement that the map on $\Ho_*(LM)$ induced by the antipodal map on $\Ss^1$ is the identity.  This follows from the fact that the antipodal map on $\Ss^1$ is homotopic to the identity map.
\end{proof}
\smallbreak
There is a natural map 
\[
\bp:LM\rightarrow M
\]
taking a free loop to its basepoint.  
\smallbreak
We take a little detour, to consider the interplay with the formal geometry around the trivial module.  We have a natural map.
\[
\res:A\otimes_{A^e} A\rightarrow k\otimes_A k
\]
This is given by a slight abuse of notation, namely we consider one copy of $k$ as a \textit{right} $A$-module, and the other as a \textit{left} $A$-module.
\begin{prop}
The following diagram commutes
\[
\xymatrix{
\HH_*(A)\ar[d]^{\res}\ar[r]^-{\cong} & \Ho_*(LM)\ar[d]^{\bp_*}\\
\Ho_*(k\otimes_A k)\ar[r]^-{\cong} &\Ho_*(M).
}
\]

\end{prop}
Now the following diagram also commutes
\[
\xymatrix{
A\otimes_{A^e} A\ar[d]\ar[r] &k\otimes_A k\ar[d]\\
\hat{A}\otimes_{\hat{A}^e} \hat{A}\ar[r]&k\otimes_{\hat{A}} k,
}
\]
and so we deduce from Lemma 11.1.2 of \cite{VdB10} that a class $\alpha\in \HH_d(A)\cong \Ho_{d}(LM)$ is nondegenerate only if $\bp_*(\alpha)$ is nonzero.  The rest of this section will be devoted to a `global' version of this result.
\medbreak
We now come back to the main thread of the paper, and to the characterisation of nondegenerate classes in $\HH_{d}(A)$.  Recall that the centre of a $k$-algebra $A$ acts on the Hochschild homology of $A$ by acting on the first copy of $A$ in $A\otimes_{A^e} A$.
\begin{prop}
\label{CYthm}
A class $\eta\in \HH_{d}(A)$ is nondegenerate if and only if it is given by $z\cdot \const_*([M])$, for $z$ a central unit of $A$.  In particular, $A$ is Calabi-Yau of dimension $d$.
\end{prop}
\begin{proof}
The result will follow from showing that $\const_*([M])$ is a nondegenerate class in $\HH_{d}(A)$.  Consider again the complex $\Gamma(S)$ of Theorem \ref{hfinite}.  Since the compliment of $\Gamma(S)$ in $\tilde{\Gamma}(S)$ is contractible, one obtains a quasi-isomorphism of chain complexes
\[
\varsigma:\Hom_{A^e}(\Gamma(S),A^e)_{(g)}\rightarrow \Ch_{\cpct}^*(\mathcal{U}S).
\]
Now the complex on the right calculates the cohomology with compact support of $\mathcal{U}S$.  From Lemma \ref{contsimp} we deduce that this has nontrivial homology only in degree $d$, and we deduce that $\Gamma(S)^{\vee}[d]$ is quasi-isomorphic to $A$ (using Poincar\'{e} duality).  Precisely, this quasi-isomorphism is given by sending the natural generator of the degree $d$ cohomology of the above complex to $g\in A$.  Note that this defines a bimodule map precisely in the case in which $M$ is orientable.  It is straightforward to check that this is in fact a bimodule quasi-isomorphism.  This establishes that $A$ is Calabi-Yau.\smallbreak
Recall that $\HH_*(A)$ has a decomposition indexed by conjugacy classes in $\pi_1(M)$.  We have chosen our isomorphism $A^{\vee}[d]\rightarrow A$ to be a \textit{graded} isomorphism.  It follows that our nondegenerate class lies in the component $\Ho_*((LM)_0)$.  Since it is nondegenerate, we deduce from Proposition \ref{hequivalence} that our nondegenerate class is some scalar multiple of $\const_*([M])$ (this can also be seen directly, by playing with chain models - see Proposition \ref{Jonloop}).
\smallbreak
Let $\eta\in \HH_{d}(A)$ be a different nondegenerate class.  Then we have a bimodule isomorphism
\[
(\const_*([M])^+)^{-1}\circ (\eta^+):A\rightarrow A.
\]
(See Definition \ref{plusnot} for the notation $-^+$.)  This is necessarily given by multiplication by a central element $z$ (since $A$ is unital).  This is the $z$ appearing in the first statement of the theorem.

\end{proof}
\begin{examp}
\label{nono}
The above proof works for non-orientable manifolds, if we allow $\mathrm{char}(k)=2$.  Let us consider the example given by setting our manifold $M$ to be the Klein bottle.  In this case $k[\pi_1(M)]=k\langle a^{\pm 1},b^{\pm 1}\rangle/(aba-b)$, and splicing the noncommutative cotangent complex with the surjection given by multiplication, we obtain the exact sequence
\[
\xymatrix{
0\ar[r] &(aba-b)/(aba-b)^2\ar[r]^-{\rho} &A\otimes_F\Omega_k^1 F\otimes_F A\ar[r] &A\otimes A\rightarrow A\ar[r] &0
}
\]
where $F=k\langle a^{\pm 1},b^{\pm 1}\rangle$.  It is not too hard to show that $(aba-b)/(aba-b)^2$ is isomorphic to the free bimodule.  So we see that $A$ is homologically finite.  We can calculate $\Ext^{2}_{A\mathrm{-bimod}}(A,A\otimes A)$ using this sequence, and we see that it is given by the bimodule $\Hom_{A\mathrm{-bimod}}(A\otimes A,A\otimes A)$, where the bimodule structure is given as ever by the inner bimodule structure on the target bimodule, modulo the homomorphisms factoring through $\rho$.  This bimodule is of course given by $A\otimes A$ itself, with the inner bimodule structure, modulo the relations $ara=r$ and $rba=-abr$.  One easily confirms, then, that a basis for this bimodule is given by elements of the form $r\otimes 1$, for $r\in A$.  $A$ acts on the right via $(r\otimes 1)s=rs\otimes 1$, for all $r,s\in A$, and on the left via $a(r\otimes 1)=a^{-1}r\otimes 1$ and $b(r\otimes 1)=-br\otimes 1$.  One can write an arbitrary element of $A$ in the form $bc$, for $c\in A$, and in terms of this presentation, we have $abc=ba^{-1}c$, from which one deduces that if $\mathrm{char}(k)=2$ the bimodule $\Ext^{2}_{A\mathrm{-bimod}}(A,A\otimes A)$ is isomorphic to the diagonal bimodule, with isomorphism given by the change of basis matrix provided by left multiplication by $b$.  Now if $\mathrm{char}(k)\neq 2$ note that $1\otimes 1$ is given by $1/2((b^{-1}\otimes 1)\cdot b - b\cdot(b^{-1}\otimes 1))=1\otimes 1$, from which we deduce that the bimodule $M$ generated by elements of the form $rm-mr$, for $r\in A$ and $m\in \Ext^{2}_{A\mathrm{-bimod}}(A,A\otimes A)$ is the whole of $\Ext^{2}_{A\mathrm{-bimod}}(A,A\otimes A)$.  So clearly $\Ext^{2}_{A\mathrm{-bimod}}(A,A\otimes A)/M$ is not isomorphic to $A_{\mathrm{ab}}$, the Abelianization of $A$, considered as an $A$-bimodule, and we deduce that $\Ext^{2}_{A\mathrm{-bimod}}(A,A\otimes A)$ is not the diagonal bimodule, and $A$ is not Calabi-Yau.
\end{examp}


\section{Main results}
\subsection{A topological obstruction}
To obtain topological obstructions to being able to write group algebras as superpotential algebras, or indeed GDGAs, we need a topological description of cyclic homology.  This is given by
\begin{thm}[\cite{Jones87} Theorem B]
Let $A$ be the fundamental group algebra for an acyclic manifold $M$.  There is an isomorphism of graded vector spaces
\[
\HC_{*}(A)\cong \Ho_{*}^{\Ss^1}(LM).
\]
Furthermore, this induces an isomorphism between long exact sequences
\begin{equation}
\label{throw}
\xymatrix{
\ar[r]&\HC_{n}(A)\ar[r]\ar[d] &\HC_{n-2}(A)\ar[r]\ar[d] &\HH_{n-1}(A)\ar[r]\ar[d] &\HC_{n-1}(A)\ar[d]\ar[r]&\\
\ar[r]&\Ho^{\Ss^1}_n(LM)\ar[r]&\Ho^{\Ss^1}_{n-2}(LM)\ar[r]^B&\Ho_{n-1}(LM)\ar[r]&\Ho^{\Ss^1}_{n-1}(LM)\ar[r]&
}
\end{equation}
\end{thm}
\begin{proof}
This is essentially done in \cite{Jones87}.  One uses the explicit chain models for equivariant and ordinary homology of the loop space constructed there, and then mimics the usual construction of the Connes long exact sequence.
\end{proof}
Recall that the bottom row of (\ref{throw}) is obtained from the Gysin long exact sequence (\ref{Gysin}) for the fibration
\[
\Ss^1\rightarrow \Ss^{\infty}\times LM\rightarrow \Ss^{\infty}\times^{\Ss^1} LM.
\]
One deduces that $B$ is given as follows: let $\alpha\in \Ho_{n-1}(\Ss^{\infty}\times^{\Ss^1} LM)$ be an equivariant homology class.  Then since $\Ss^{\infty}\times^{\Ss^1} LM$ is an orbit space, we obtain a class $\overline{\alpha}$ in $\Ho_n(\Ss^{\infty}\times LM)$.  Finally we have that $B\alpha=p_*(\overline{\alpha})$, where 
\[
p:\Ss^{\infty}\times LM\rightarrow LM
\]
is the natural projection.  
\medbreak
Note that the path components of $LM$ are in bijective correspondence with conjugacy classes in $\pi_1(M)$.  If $c$ is a conjugacy class in $\pi_1(M)$ we define $(LM)_c$ to be the component corresponding to $c$.
\begin{rem}
\label{coverings}
For an arbitrary $k$-algebra there is a natural contraction pairing $\HH^s(A)\times \HH_t(A)\rightarrow \HH_{t-s}(A)$, where $\HH^*(A)$ is the Hochschild cohomology of $A$.  In the present case we will only be interested in the action of $\HH^0(k[\pi_1(M)])=\Zz(k[\pi_1(M)])$.  For example, let $z$ be a central element of $k[\pi_1(M)]$, of the form
\[
z=\sum_{i\in I} g_i
\]
where the $g_i$ are all conjugate.  We will be interested in calculating $\const_{*}([M])\cdot z$.  We can of course express this in terms of our chain model for the loop space, but it is instructive to see what this is geometrically.  Let $C_{\pi_1(M)}(g)$ be the centralizer of $g$.  Then there is a finite cover
\[
\xymatrix{
\tilde{M}\ar[d]\\
M,
}
\]
corresponding to the subgroup $C_{\pi_1(M)}(g)$ under the usual Galois correspondence for covering spaces.  Note that $g$ lifts to a central element of $\pi_1(\tilde{M})$.  It is possible to give a map $\tilde M\rightarrow (L\tilde{M})_{g}$, which is a section of 
\[
\tilde{\textbf{bp}}:L\tilde{M}\rightarrow \tilde{M}
\]
Composing this with the obvious map $L\tilde{M}\rightarrow LM$ we obtain the desired class in $\HH_*(k[\pi_1(M)])$.  
\end{rem}

\begin{thm}
Let $M$ be a $d$-dimensional compact acyclic manifold.  Then $k[\pi_1(M)]$ is exact Calabi-Yau if and only if there is a central unit $\chi$ in $k[\pi_1(M)]$, of the form
\[
\chi=\sum_{c\in C}a_c s_c
\]
where $C$ is a finite set of finite conjugacy classes of $G$, and $s_c:=\sum_{g\in c} g$, and for each $c\in C$ there is a class $\lambda_c\in \Ho_{d-1}^{\Ss^1}((LM)_c)$ such that $\bp_*(B(\lambda_c))\neq 0$.
\end{thm}
\begin{proof}
The form given for a central unit in $k[\pi_1(G)]$ is clearly the form that any central element of $k[\pi_1(M)]$ must take.\smallbreak
Let $f:A^{\vee}[d]\rightarrow A$ be an isomorphism.  We have already an isomorphism $f':A^{\vee}[d]\rightarrow A$ given by the cycle $\const_*([M])\in \Ho_{d}(LM)$, by Proposition \ref{CYthm}.  Therefore $f\circ f'^{-1}$ is a bimodule isomorphism $A\rightarrow A$, and so it must be given by multiplication by a central unit.  Let $\gamma_c\in \Ho_{d}(LM)$ be an irreducible cycle.  This forces $\gamma_c=\const_*([M])\cdot a_c s_c$ for some $a_c$.  Finally, again from the explicit chain model, one reads off that
\[
\bp_*(\const_*([M])\cdot a_c s_c)=|c| a_c [M]
\]
where $|c|$ is the number of elements of $c$.  Now $\Ho^{\Ss^1}_*(LM)$ comes also with a decomposition indexed by conjugacy classes of $\pi_1(M)$, and the map $B$ respects this decomposition.  We deduce that there must exist a $\lambda_c$ as in the statement of the theorem.
\end{proof}
\begin{rem}
In the present context, there is a slight ambiguity in the phrase `nontrivial unit'.  Group theorists refer to all elements $g\in k[G]$ of a group ring as trivial, while here we only consider $1\in k[G]$ as trivial.
\end{rem}
\begin{thm}
\label{biggun}
Let $M$ be a $d$-dimensional compact acyclic orientable manifold.  Then $k[\pi_1(M)]$ is an exact Calabi-Yau algebra only if it contains a nontrivial central unit.
\end{thm}
\begin{proof}
This will follow from the fact (which we now prove) that $\const_*([M])$ is not an exact Calabi-Yau structure, and the fact that any two Calabi-Yau structures differ by a central unit.  Recall Proposition \ref{hequivalence}, which states that
\[
\const:M\rightarrow (LM)_0
\]
is a homotopy equivalence, which is also $\Ss^1$-equivariant after giving $M$ the trivial action, and so it induces an isomorphism
\[
\const^{\Ss^1}_*:\Ho_{*}^{\Ss^1}(M)\rightarrow \Ho_{*}^{\Ss^1}((LM)_0).
\]
By construction there is an isomorphism
\[
\Ho_{*}^{\Ss^1}(M)\cong \Ho_*(M\times \Cp \mathbb{P}^{\infty}).
\]
Composing these isomorphisms with the map $B$ we obtain a map
\[
\Gamma:\Ho_{d-1}(M\times \Cp \mathbb{P}^{\infty})\rightarrow \Ho_{d}(M)
\]
taking a cycle $a\in \Ho_{d-1}(M\times \Cp \mathbb{P}^{\infty})$ to a cycle $b$ in $\Ho_{d}(M)$ whose support is contained in the support of $p_*(a)$, where 
\[
p:M\times \Cp \mathbb{P}^{\infty}\rightarrow M
\]
is the natural projection.  We deduce that $b$ is zero, and so in particular, $\const_*([M])$ is not an exact Calabi-Yau structure.  It follows that the central unit one obtains by comparing an isomorphism $f:A^{\vee}[d]\rightarrow A$ obtained from an exact Calabi-Yau structure on $A$ with the one obtained from $\const_{*}([M])$ must be nontrivial.
\end{proof}
\begin{cor}
\label{crux}
Let $M$ be a $d$-dimensional compact acyclic manifold.  Then $k[\pi_1(M)]$ is a GDGA or a superpotential algebra only if it contains a nontrivial central unit.
\end{cor}
\begin{proof}
Combine Theorems \ref{biggun} and \ref{thm1}.
\end{proof}
\begin{rem}

For a general compact orientable manifold $M$ one expects (from consideration of Koszul duality) that the $A_{\infty}$-algebra $\Ch_{*}(\Omega M)$ is homologically finite and Calabi-Yau.  In the case $M=\Ss^n$ this algebra is in fact a superpotential algebra (see Example \ref{spheres}).  This demonstrates the failure of the above Theorem in this case, since $\pi_1(\Ss^n)$ has no nontrivial units.  This failure can be traced back to the failure of Proposition \ref{hequivalence} in this case.  Nonzero cycles in the image of the map $\Ho_{d-1}^{\Ss^1}(LM)\rightarrow \Ho_d(LM)$ can be identified with morphisms of nonzero degree from the total spaces of circle bundles over $(d-1)$-dimensional bases to $M$.  In the case $\Ss^{2k+1}$ these exist since these spheres actually \textit{are} circle bundles (e.g. the Hopf fibration).  In the case of even-dimensional spheres it is still straightforward to find the required maps from circle bundles: the $n$-dimensional torus is a trivial circle bundle obtained by taking the (filled) hypercube $[0,1]^n$ and identifying some points on the boundary, while the $n$-dimensional sphere is obtained by identifying \textit{all} points on the boundary of the same hypercube.
\end{rem}

\begin{conj}
\label{dodg}
Let $M$ be a $d$-dimensional compact orientable acyclic manifold.  Then the classes in $\HH_d (k[\pi_1(M)])$ corresponding to exact Calabi-Yau structures are given precisely by central units with zero constant coefficient.
\end{conj}
To prove the conjecture, it is enough to establish the following geometric claim: Given a compact orientable $d$-manifold $M$, and $z$ a nontrivial central element of its fundamental group, there is a circle bundle over $N$, a $(d-1)$-dimensional manifold, with total space $E$, and a map $E\rightarrow M$ with nonzero degree, mapping the fibre $\Ss^1$ to $z$.  One reduces the problem to that of settling this claim by considering covering spaces (see Remark \ref{coverings}).  The question of whether this claim is true seems to be interesting in its own right.
\subsection{Counterexamples from hyperbolic geometry}
We have reduced the task of finding counterexamples to Conjecture \ref{mfldconj} (and also, barring friendliness issues the conjecture mentioned in Remark \ref{con2}) down to some group theory, i.e. we need to find a group $G$ such that $BG$ is homotopic to a compact orientable manifold and $k[G]$ has no nontrivial central units.  Unfortunately, as this stands, it is still difficult to classify counterexamples, as this particular area of group theory appears to be a morass of very difficult open problems.  So in order to make progress we restrict to a class of groups for which a little more is known.  In this section we parachute in some hyperbolic geometry.  So first, the definition:
\begin{defn}
A hyperbolic manifold is a compact Riemannian manifold with constant negative sectional curvature -1.
\end{defn}
We restrict attention to \textit{orientable} hyperbolic manifolds.  These are obtained from quotienting hyperbolic space by orientation-preserving isometries.  The universal cover of such an object is contractible, since it is just hyperbolic space.  It follows by Proposition \ref{CYthm} that if $M$ is a hyperbolic manifold, $k[\pi_1(M)]$ is Calabi-Yau of dimension $\dim(M)$.  The fundamental group of a hyperbolic manifold is, in some sense, the archetypal example of a \textit{hyperbolic group}.  Everything needed for this paper regarding hyperbolic groups can be found in \cite{Howienotes}.  The following lemma is what enables us to find counterexamples.
\begin{lem}
Let $g\in G$ be an element of a hyperbolic group.  Then the centralizer of $g$ contains $\langle g\rangle$ as a finite index subgroup.
\end{lem}
\begin{cor}
\label{hypform}
Let $M$ be a hyperbolic manifold of dimension greater than 1.  Then
\begin{enumerate}
\item
The centre of $\pi_1(M)$ is trivial.
\item
The centre of $k[\pi_1(M)]$ is trivial.
\end{enumerate}
\end{cor}
\begin{proof}
It is obviously enough to prove the second assertion.  From our previous description of central elements of $k[\pi_1(M)]$, it is necessary and sufficient to prove that all conjugacy classes in $\pi_1(M)$ are infinite. 
\smallbreak
For a contradiction, assume that $c$ is a conjugacy class with finitely many elements.  Then we obtain a group homomorphism $\pi_1(M)\rightarrow S_{|c|}$, the permutation group on $c$, via the conjugation action.  The kernel of this homomorphism is of finite index, and is contained in the centralizer of $g$, an arbitrary element of $c$.  It follows that the centraliser of $g$ is a finite index subgroup of $\pi_1(M)$, and so by the lemma, $\pi_1(M)$ contains $\langle g\rangle$ as a finite index subgroup.  Let
\[
\xymatrix{
M'\ar[d]\\
M
}
\]
be the cover corresponding to the subgroup $\langle g\rangle$.  It follows that it is a finite cover, and so $M'$ is compact.  Since $M'$ is a $BG$ for $\mathbb{Z}$, it follows that there is a homotopy equivalence $M'\rightarrow \Ss^1$.  We deduce that $M'$ is a compact orientable manifold with the same cohomology as the circle, and so it is the circle, a contradiction of our assumption on the dimension of the manifold $M$.
\end{proof}
\begin{cor}
Let $M$ be a compact hyperbolic manifold of dimension greater than 1.  Then $k[\pi_1(M)]$ is a Calabi-Yau algebra that is not a GDGA, and in particular not a superpotential algebra.
\end{cor}
\begin{proof}
We apply Corollaries \ref{hypform} and \ref{crux}.
\end{proof}
\section{Future directions}
\subsection{Positive results}
The main result of this paper appears to be negative, i.e. it establishes that one cannot give superpotential descriptions of some fundamental group algebras.  We try here to undo a little of the damage, by giving a weaker conjecture and some consideration of it in low dimensions.
\begin{conj}
If $M$ is an acyclic compact orientable manifold, then $k[\pi(M)]$ is a superpotential algebra if it is a circle bundle.
\end{conj}
In 2 dimensions there is only one non-hyperbolic 2-dimensional acyclic compact orientable manifold.  It is easy to find a superpotential description for this manifold.\smallbreak
In 3 dimensions it is a little harder.  We establish the positive result for trivial circle bundles $M=N\times \Ss^1$.  Acyclicity forces $N$ to be a surface of genus at least 1.  In the case $M=(\Ss^1)^3$ this is just Example \ref{3-torus}.  In higher genus we use the language of \textit{brane tilings}.  These are covered in the higher genus case in \cite{Dav08}, see also \cite{MR}, and \cite{longout}.  We won't recall the definitions here.  Given a consistent (in the sense of \cite{Dav08}) brane tiling we obtain a pair $(\Cp Q,W)$ of a free quiver algebra and a $W\in (\Cp Q)_{\cyc}$, a linear combination of cyclic words in the quiver.  We obtain a new quiver by contracting a maximal tree, and removing any instances of contracted edges from words in $W$.  This gives us a new superpotential algebra.  Finally, we obtain the superpotential algebra we want by inverting all the arrows of the modified quiver.  So it is enough to find a consistent brane tiling for a surface of genus $g\geq 2$.  This we do as follows.  Take two copies of the tile in Figure \ref{genustile}, and glue them along the thick edges, obtaining a 2-manifold with boundary given by a number of circles.  Taking the mirror image of the tile in Figure \ref{genustile} we obtain a new tile, and taking two copies of this tile, glued again along the thick edges, we obtain another 2-manifold with boundary.  Gluing these two 2-manifolds together, we obtain a surface, and a tiling of it, that one can easily check is consistent.
\begin{figure}

\centering
\input{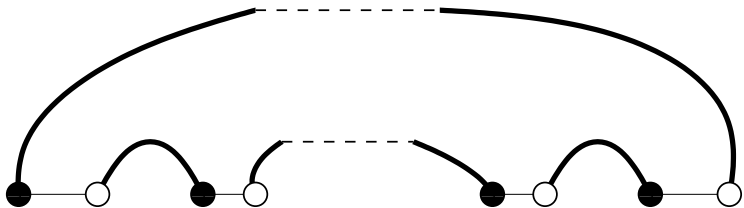}
\caption{One of four tiles with which we tile a higher genus surface}
\label{genustile}
\end{figure}

\begin{examp}
\label{gen2}
Following this procedure in genus 2, we obtain the pair $(Q_2,W_2)$.  where $Q_2$ is a quiver with 1 vertex, and arrows labelled $a,...,i$, and 
\[
W_2=fi-gia+adg-bdh+behc-cef.
\]
It is easy to see that appropriate modifications of the quiver and of $W_2$ enable us to remove quadratic terms.  However, in this example it is more useful to leave them in (see next subsection).
\end{examp}

\subsection{Donaldson-Thomas invariants of 3-manifolds}
One of the principal motivations for undertaking the study of the question of when a manifold admits a superpotential description was a desire to understand, and compute, motivic Donaldson-Thomas invariants for them (as defined in \cite{KS}).  See also the final section of \cite{COHA} for a brief discussion of this problem in (the only!) easy cases.  Roughly speaking, the motivic Donaldson-Thomas invariant of a moduli space is a motivic refinement of the topological Euler characteristic, weighted by Behrend's microlocal function (see \cite{Behr09}, \cite{ObstructionsHilbert}).  There are two distinct moduli problems one might attempt to associate these motivic invariants to: the problem of parameterising perfect modules over the $A_{\infty}$-enriched cohomology algebra of a manifold $M$, and the problem of parameterising finite-dimensional representations of $\pi_1(M)$.
\smallbreak
We focus on the second problem.  The \textit{point} of finding a superpotential description, in this case, is that it gives a description of the moduli space of representations of $\pi_1(M)$ as the critical locus of a function $f$ on a smooth scheme.  The smooth scheme is the representation variety of the smooth algebra appearing in the description of $\pi_1(M)$ as a superpotential algebra, and the function is given by $trW$.
\smallbreak
Given such a description we are in much better shape for actually doing some calculations.  In the case $(\Ss^1)^3$, the answer can be deduced from basic manipulations of the answer given in \cite{BBS}, which deals with the Hilbert space of $\Cp^3$.  This relies heavily on the algebra $\Cp[x,y,z]$ having a superpotential description of a special kind: let $(B,W)$ be a pair of a smooth algebra and superpotential giving rise to our algebra $A=k[\pi(M)]$.  Then we wish to find a torus action on $B$ such that $W$ is homogeneous of weight 1.
\smallbreak
Returning to Example \ref{gen2}, one quickly shows that this is possible for the given superpotential.  However, if we modify the algebra by deleting the edge $f$, deleting the $fi$ term of $W$ and replacing $cef$ with $cei$, we obtain a pair that does \textit{not} satisfy this property, although it is a simpler superpotential description of the fundamental group algebra.  The general recipe we have given for obtaining superpotential descriptions of 3-manifolds given by trivial circle bundles over acyclic surfaces always produces a superpotential satisfying the desired property regarding torus weights.

\subsection{Some remarks on topological field theory}
Let $A=k[\pi_1(M)]$ for $M$ an acyclic compact orientable $d$-manifold.  We start this subsection by revisiting Proposition \ref{flip}.  This states that the flip isomorphism on Hochschild homology is actually the identity, since it is given, at the level of chains on $LM$, by rotating loops by $\pi$.  Now we can of course promote the question of finding fixed points of $\beta$ to a question in dg-categories, if we take a dg-model for the category of bimodules over $A$, at which point we can ask a more refined question, we may ask that $\beta$ fixes our nondegenerate class at the chain level.  If this condition is satisfied we say that $\eta$ is a \textit{chain level} Calabi-Yau structure for $A$.
\smallbreak
If we take as our chain model for Hochschild homology the free loop space, then this requirement has an obvious geometric meaning.  Note that the $\mathbb{Z}_2$-action here is really a restriction of a $\Ss^1$-action, and furthermore, by construction, any nondegenerate class $\eta\in \Zz_d(C(LM))$ in the image of the $B$-map will be automatically $\Ss^1$-invariant.  The $\Ss^1$-invariant, nondegenerate classes in $\Zz_d(\Ch_*(LM))$ play a special role in topological field theory, namely (see the final section of \cite{Lur09}) they correspond exactly to 2-dimensional extended topological field theories.  In short, these are functors from a higher bordism category, the i-morphisms of which are given by i-manifolds with boundary, for $i\leq 2$, and the higher morphisms of which are given by diffeomorphisms, isotopies between diffeomorphisms, etc.  These functors land in a higher category whose objects are algebras, whose morphisms are bimodules, and whose two-morphisms are bimodule maps, with higher structure given by the calculation of $\mathrm{RHom}$s for bimodules.
\smallbreak
One may summarise the situation as follows: for a class $\eta\in \Zz_d(\Ch_*(LM))$ we have
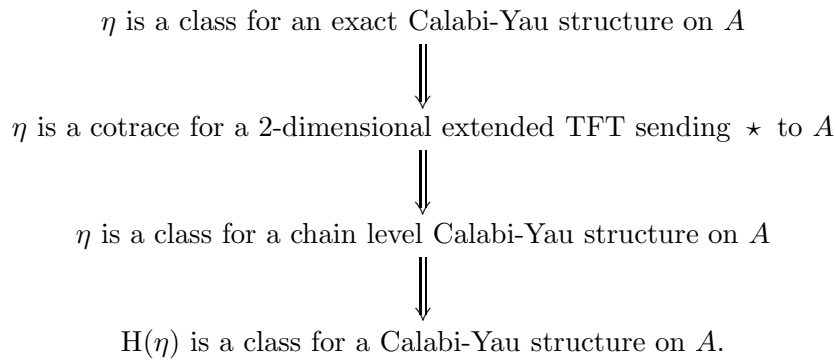
\[
\xymatrix{
\eta \hbox{ is a class for an exact Calabi-Yau structure on } A\ar@{=>}[d]\\
\eta \hbox{ is a cotrace for a 2-dimensional extended TFT sending } \star \hbox{ to } A\ar@{=>}[d]\\
\eta \hbox{ is a class for a chain level Calabi-Yau structure on } A\ar@{=>}[d]\\
\Ho(\eta) \hbox{ is a class for a Calabi-Yau structure on } A.
}
\]
All of these implications are strict (the first one thanks to the fact that $\const_*([M])$ is not an exact Calabi-Yau structure on $A$).

\bibliographystyle{amsplain}
\bibliography{SPP}
\end{document}